\documentclass[11pt, reqno, a4paper]{amsart}
\usepackage{amssymb,latexsym,amsmath,amsfonts,bbm, amsthm}
  \usepackage{paralist}
  \usepackage{graphics} 
  \usepackage{epsfig} 
\usepackage{graphicx}  \usepackage{epstopdf}
\usepackage[linkcolor=red, citecolor=magenta]{hyperref}
\hypersetup{colorlinks=true, urlcolor=blue}
\usepackage[nameinlink]{cleveref}
\numberwithin{equation}{section}
 \usepackage{comment}
\usepackage{stmaryrd}
\usepackage{subcaption} 
\usepackage{pgfplots}
\usepackage{pgfplotstable}
\pgfplotsset{compat=1.18}
\newcommand{\norm}[1]{\left\|#1\right\|}

\newcommand{\ix}{\int_{0}^{x}}

\newcommand{\tl}{\tilde}

\newcommand{\N}{{\mathbb{N}}}

\numberwithin{equation}{section}

 \newtheorem{theorem}{Theorem}[section]

 \newtheorem{cor}[theorem]{Corollary}
 \newtheorem{defn}[theorem]{Definition}

 \newtheorem{lem}[theorem]{Lemma}
 \newtheorem{prop}[theorem]{Proposition}
 \newtheorem{rem}[theorem]{Remark}
 \newcommand{\pa} {\partial}

 \newcommand\dis{\displaystyle}
 \newcommand\inter[1]{\llbracket #1\rrbracket}
  
 \usepackage{algorithm}
\usepackage{algpseudocode}

 \renewcommand{\leq}{\leqslant}

 \def\hmath$#1${\texorpdfstring{{\rmfamily\textit{#1}}}{#1}}
 
 \newcommand{\mc}{\mathcal}
 \def\textmatrix#1&#2\\#3&#4\\{\bigl({#1 \atop #3}\ {#2 \atop #4}\bigr)}
 \def\dispmatrix#1&#2\\#3&#4\\{\left({#1 \atop #3}\ {#2 \atop #4}\right)}
 \theoremstyle{remark}
 \newsavebox{\savepar}		 
 
 \usepackage{xcolor}

\date{\today}
\begin{document}
\title[Event-triggered control for FHN]{Event-triggered boundary control of the linearized FitzHugh-Nagumo equation$^{*}$}

\author[Víctor Hernández-Santamaría]{Víctor Hernández-Santamaría$^{\ddagger}$}
\author[Subrata Majumdar]{Subrata Majumdar$^\dagger$}
\author[Luz de Teresa]{Luz de Teresa}
\thanks{ Instituto de Matem\'{a}ticas, Universidad Nacional Aut\'{o}noma de M\'{e}xico, Circuito Exterior C.U., 04510, CDMX, Mexico (emails: \texttt{victor.santamaria@im.unam.mx}, \texttt{subrata.majumdar@im.unam.mx}, \texttt{ldeteresa@im.unam.mx})}
\thanks{$^{*}$This work has received support from Project A1-S-17475 of CONAHCYT, Mexico.}
\thanks{$^{\ddagger}$V. Hern\'andez-Santamar\'ia is supported by the program ``Estancias Posdoctorales por México para la Formación y Consolidación de las y los Investigadores por México'' of CONAHCYT (Mexico). He also received support from Project CBF2023-2024-116 of CONAHCYT and by UNAM-DGAPA-PAPIIT grants IN109522, IA100324 and IN102925 (Mexico). }
\thanks{$^\dagger$Subrata Majumdar is supported by the UNAM Postdoctoral Program (POSDOC)}
\keywords{FitzHugh-Nagumo equation, Backstepping method, Event-triggered control}
\subjclass[2020]{}

\medskip

\begin{abstract}
In this paper, we address the exponential stabilization of the linearized FitzHugh-Nagumo system using an event-triggered boundary control strategy. Employing the backstepping method, we derive a feedback control law that updates based on specific triggering rules while ensuring the exponential stability of the closed-loop system. We establish the well-posedness of the system and analyze its input-to-state stability in relation to the deviations introduced by the event-triggered control. Numerical simulations demonstrate the effectiveness of this approach, showing that it stabilizes the system with fewer control updates compared to continuous feedback strategies while maintaining similar stabilization performance.
\end{abstract}
	\maketitle


\section{Introduction}

\subsection{Motivation} In control systems, particularly when implemented on digital platforms, a key challenge is balancing performance with resource usage. Traditional digital control schemes, such as sampled-data control, typically rely on periodic updates of the control signal. However, this can lead to unnecessary computational effort, especially when the system state evolves slowly. This issue is even more critical in networked control systems, where excessive communication can exhaust available resources (see, for example, \cite{heemels2012introduction} for a more detailed introduction in the finite-dimensional case, and \cite{KK18,KFS19} for applications in the parabolic setting).

Event-triggered control has emerged as an effective alternative, performing updates only when a predefined condition (based on the system state) is fulfilled. This approach allows the system to remain stable while minimizing control updates and preserving resources. The approach has shown promising results in finite-dimensional systems modelled by ordinary differential equations (ODEs) \cite{tabuada2007event, postoyan2014framework} and, more recently, in systems described by partial differential equations (PDEs) \cite{espitia2021event,LB, koudohode2022event,baudouin2023event,WJ,KEK24}.

In the context of PDEs, the most common stabilization results rely on controlling from the boundary of the equation. The backstepping technique (see the seminal work \cite{krstic2008boundary}) offers a systematic approach for designing stabilizing feedback laws applicable to a wide range of systems. One of its key advantages is that this method is robust and adaptable across various frameworks, having been successfully applied to numerous equations and models while ensuring global stability. We refer to the non-exhaustive list \cite{CN17,GLM21,CHXZ22,Aur24,LM24,dAVKK24,HL24,PA24} and the references therein for recent progress and works in this direction.

Building on this framework, we consider the application of event-triggered control to a system known as the FitzHugh-Nagumo model. We employ the backstepping feedback as a starting point and apply the event-triggered control strategy to a linearized version of the model. As established in prior work (see \cite{chowdhury2024local}), this linearization possesses specific stabilization properties, most notably the restriction that stabilization can only be achieved for decay parameters determined by the system's coefficients rather than arbitrary values. We extend this analysis to the event-triggered setting, demonstrating that the system remains stable and that the stabilization rate can be made arbitrarily close to that of the continuous control case.

\subsection{Setting of the problem}\label{setting of the problem}

Let us consider the following nonlinear coupled ODE-PDE reaction-diffusion system in the interval $(0,1)$ with non-monotone non-linearity $I_{\text{ion}}(\cdot,\cdot)$ of FitzHugh-Nagumo type 
\begin{equation}\label{ODEPDE}
	\begin{cases}
		\pa_t v=\pa_{x}^2v+I_{\text{ion}}(v, w) & \text{ in }   (0,\infty) \times (0,1),\\
		\pa_tw=\gamma v-\delta w & \text{ in }   (0,\infty) \times (0,1),\\
		v(t,0)=0, \ \ v(t,0)=q(t) &\text{ in }   (0,\infty),\\
		v(0,x)=v_0(x), \ w(0,x)=w_0(x)\ \ & \text{ in }   (0,1).
	\end{cases}
\end{equation}
In \eqref{ODEPDE}, $v=v(x,t)$ and $w = w(x,t)$ are the state variables while $q \in L^2(0,\infty)$ is a boundary control. The nonlinearity is of the form:
\begin{align*}
	I_{\text{ion}}(v, w)=-v(v-1)(v-a)-\rho w.\label{eq:FHNn}
\end{align*}
Here, the system parameters are $a \in (0,1)$, and $\gamma$, $\delta$, and $\rho$ are positive constants.

We refer the reader to the review article \cite{HS75} for a comprehensive study of this mathematical model in neurobiology. This model describes the conduction of electrical impulses in a nerve axon (see also \cite{RM94}) and is also known in cardiac electrophysiology as the monodomain equations (see \cite{PBW91}).

In this paper, we are interested in studying the exponential stabilization by means of an event-triggered control of the following linearized version (around the origin) of \eqref{ODEPDE} 
\begin{equation}\label{FHNlin}
	\begin{cases}
		\pa_t v=\pa_{x}^2v-av-\rho w&\text{ in } (0,\infty) \times (0,1),\\
		\pa_t w=\gamma v-\delta w &\text{ in }  (0,\infty) \times (0,1),\\
		v(t,0)=0,\,\, v(t,1)=q(t)& \text{ in } (0,\infty),\\
		v(0,x)= v_0(x) \quad w(0,x)= w_0(x)& \text{ in } (0, 1).
	\end{cases}
\end{equation}
We begin by recalling the notion of exponential stabilization by feedback control.
\begin{defn}\label{def:exp}
 Let $(v_0, w_0)\in L^2(0,1)\times L^2(0,1)$. The system \eqref{FHNlin} is called exponential stabilizable by feedback control in the space $L^2(0,1)\times L^2(0,1)$ with a decay rate $\omega>0,$ if there exists a  bounded linear map $\Pi: L^2(0,1)\times L^2(0,1) \to \mathbb{R}$ such that, the solution $(v,w)$ of \eqref{FHNlin} with control of the form $q(t)=\Pi(v,w)(t,\cdot)$
 satisfies
	$$ 	\norm{v(t)}_{L^2(0,1)}+\norm{w(t)}_{L^2(0,1)}\leq C e^{-\omega t}\left(\norm{v_0}_{L^2(0,1)}+\norm{w_0}_{L^2(0,1)}\right),$$ for all $t>0$, for some positive constant $C$ independent of $u_0$ and $w_0$.
\end{defn}
 It is well-known that the system \eqref{FHNlin} is not exponentially stabilizable in the sense of \Cref{def:exp} with a decay rate $e^{-\omega t}$ if $\omega>\delta$ (see \cite[Corollary 3.13]{chowdhury2024local}). This limitation arises due to the presence of the ODE component in \eqref{FHNlin}, which acts as a memory term, preventing any improvement in the decay rate beyond the critical value $\delta$. See \cite{Ch12}, \cite{Chowdhury2}, and \cite{Kkjmpa} for this kind of obstruction in related control problems.

Despite the obstruction to exponential stabilization when $\omega>\delta$, feedback control still plays a role in enhancing the system's stability with decay rate $\omega\leq \delta$. By direct computations, it can be checked that if $q=0$, system \eqref{FHNlin} is exponentially stable with decay rate $\omega=\min\{a, \delta\}$. On the other hand, it has been proved in \cite{chowdhury2024local} that using the method of backstepping, we can go up to the decay $e^{-\omega t}$ for any $\omega\leq \delta$. According to these facts, if $\delta < a$, system \eqref{FHNlin} with $q = 0$ already exhibits the expected exponential stability. Consequently, the case $\delta > a$ is the most relevant for the design of the feedback control law, which is why we restrict our analysis to this scenario in the present paper.

The stabilization result in \cite{chowdhury2024local} is established by the  action of an explicit feedback control law of the form
\begin{equation}
\label{feedback1}
q(t)=v(t,1)=\int_{0}^{1}k(1,y)v(t,y)dy, \quad t>0,
\end{equation}
where $k$ is a suitable kernel (see \cref{kernel} in \Cref{sec:backstepping}). In this paper, employing an event-triggering approach, we replace the continuous control \eqref{feedback1} by 
\begin{equation*}
q_d(t)=v(t,1)=\int_{0}^{1}k(1,y)v(t_j,y) dy,
\end{equation*}
for all $t\in [t_j, t_{j+1}), j\geq 0$, where $\{t_j\}_j$ is a sequence of triggering times obeying a well-prescribed rule (see \Cref{defn-ev} below). The price to pay for applying this piecewise-constant control is that we are unable to achieve the critical decay rate $\delta$, but only the rate
\begin{equation}
e^{-\omega t} \quad\text{where}\quad \omega=\delta-\epsilon \quad  \forall \epsilon \in (0,\delta).
\end{equation}

\subsection{Outline of the paper}The rest of the paper is organized as follows. In \Cref{sec:backstepping}, we describe a preliminary structure of the backstepping approach along with the event-triggering control strategy. \Cref{wellp} contains well-posedness results for the event-triggered control system. \Cref{exp stab} is devoted to the description of the event-triggering rule, the avoidance of the occurrence of the Zeno solution and the main result (\Cref{main_theorem}) regarding the exponential stabilization of the linearized FHN system. Finally, in \Cref{sec:5}, we conclude the paper by introducing some numerics which illustrate the result of the paper.


\section{Brief description on the structure of Backstepping under the Event-triggered strategy}\label{sec:backstepping}
Our main goal is to investigate the exponential stabilizability of the closed-loop linear FHN system \eqref{FHNlin} with an event-triggered approach. 
This method involves stabilization based on events by sampling the feedback control law \eqref{feedback1} for the continuous-time backstepping case at specific time instants, which form an increasing sequence $\{t_j\}_{j\in \N},$ along with $t_0=0$. These time instants will be characterized later through a triggering rule.  Between two consecutive time instants, the control value remains constant and is only updated when a state-dependent condition is satisfied. 

In this context, we will use the boundary control of the form 
\begin{equation}\label{con tr}
v(t,1)=q_d(t)=\int_{0}^{1}k(1,y)v(t_j,y) dy,
\end{equation}for all $t\in [t_j, t_{j+1}), j\geq 0,$ where $k$ satisfies the following wave equation in the triangle $\mathcal{T}=\{(x,y)\in [0,1]\times [0,1]: 0\leq y \leq x\leq 1\}$
\begin{equation}\label{kernel}
	\begin{cases}
		\pa_{x}^2 k(x,y)-\pa_{y}^2 k(x,y)-(\lambda-a)k(x,y)=0 & 0< y< x<1,\\
		2\dfrac{d}{dx}k(x,x)+(\lambda-a)=0& 0\leq x\leq 1,\\
		k(x,0)=0, & 0\leq x\leq 1,
	\end{cases}
\end{equation}
where $\lambda$ is a positive constant. Thanks to Lemma 2.2 of \cite{Liu03} and Lemma 4.4 of \cite{Liubook}, we see that the equation \eqref{kernel} has a unique $\mc C^2$ solution. Furthermore, we can write the solution $k$ in the following series expansion
\begin{equation*}
	k(x,y)=-\sum\limits_{n=0}^{\infty}\left(\frac{\lambda-a}{4}\right)^{n+1}
	\frac{2y(x^2-y^2)^n}{\left({n!}\right)^2(n+1)},
\end{equation*}
and thanks to \cite[Chapter 4]{Kr08}, we also have the following representation 
\begin{equation}\label{kernel bes}
k(x,y)=(a-\lambda)y\frac{\mc I_1\left(\sqrt{(a-\lambda)(y^2-x^2)}\right)}{\sqrt{(a-\lambda)(y^2-x^2)}},
\end{equation}
where $\mc I_1$ is the Bessel function given by $\mc I_1=\sum\limits_{n=0}^{\infty}\frac{(\frac{x}{2})^{1+2n}}{n! (n+1)!}.$

\noindent
Note that, the expression of the control $q_d$ in \eqref{con tr} can be written in the following fashion
\begin{equation}
q_d(t)=\underbrace{\int_{0}^{1}k(1,y)v(t,y)dy}_{q(t)}+d(t),
\end{equation}
where $d(t)$ is the difference between the event triggering control and the continuous-time backstepping control at $t\in [t_j, t_{j+1})$
\begin{equation}\label{d}
d(t)=\int_{0}^{1}k(1,y)v(t_j,y)dy-\int_{0}^{1}k(1,y)v(t,y) dy.\end{equation}
Let us formulate the concerned system with an event-triggered control strategy $\forall j\geq 0$
\begin{equation}\label{FHNlin ET}
	\begin{cases}
		\pa_t v=\pa_{x}^2v-av-\rho w&\text{ in } (t_j, t_{j+1}) \times (0,1),\\
		\pa_t w=\gamma v-\delta w &\text{ in }  (t_j, t_{j+1}) \times (0,1),\\
		v(t,0)=0,\,\, v(t,1)=q_d(t)=\int_{0}^{1}k(1,y)v(t_j,y) dy& \text{ in } (t_j, t_{j+1}),
	\end{cases}
\end{equation}
and initial conditions $	v(0,x)= v_0(x) \quad w(0,x)= w_0(x), \text{ in } (0, 1).$
Next, we employ the well-known Backstepping method for the event-triggered boundary control system \eqref{FHNlin ET}. For that, let us introduce the Volterra integral transformation of the second kind $\Pi:L^2(0,1) \to L^2(0,1)$ is defined by 
\begin{equation}\label{volterra}
	(\Pi\sigma)(x)= \sigma(x)-\ix k(x,y)\sigma(y) dy, \quad x \in [0, 1],\, \sigma \in L^2(0,1),
\end{equation}
where the kernel function $k$ is the solution of the equation \eqref{kernel}. 
Note that $\Pi : L^{2}(0,1) \to L^{2}(0, 1)$ and $\Pi^{-1} : L^{2}(0,1) \to L^{2}(0, 1)$  are linear bounded operators (see Lemma $2.4$ in \cite{Liu03}).
The inverse of the transformation \eqref{volterra} is given by 
\begin{equation}\label{inverse_kernel}
	(\Pi^{-1} \hat \sigma)(x)= \hat\sigma(x)+\ix l(x,y)\hat \sigma(y) dy, \quad x \in [0, 1],
\end{equation}
where $l$ is the solution of the following equation
\begin{equation}\label{l}
	\begin{cases}
		\pa_{x}^2 l(x,y)-\pa_{y}^2l(x,y)+(\lambda-a)l(x,y)=0 & 0< y< x< 1,\\
		2\dfrac{d}{dx}l(x,x)+(\lambda-a)=0 & 0\leq x\leq 1,\\
		l(x,0)=0 & 0\leq x\leq 1.
	\end{cases}
\end{equation}
Using the transformation \eqref{volterra}, let us define 
\begin{equation}\label{ch var}
	\begin{cases}
		u(t,\cdot)=\Pi v(t,\cdot),\\
		z(t,\cdot)=\Pi w(t,\cdot).
	\end{cases}
\end{equation}
Next, one can show that the system \eqref{FHNlin ET} can be transformed to the following target system
\begin{equation}\label{FHNtargetvariable}
\begin{cases}
	 \pa_t u-\pa_{x}^2u+\lambda  u+ \rho z =0&\text{ in } (t_j, t_{j+1}) \times (0,1),\\
	\pa_t z=\gamma u-\delta z &\text{ in } (t_j, t_{j+1}) \times (0,1),\\
	u(t,0)= 0, \quad u( t,1)=d(t) & \text{ in } (t_j, t_{j+1}),\\
	 u(0,x)= u_0(x), \quad  z(x, 0)= z_0(x)& \text{ in } (0,1).
\end{cases}
\end{equation}
where $\lambda>0$ is called the damping parameter, to be chosen as sufficiently large, $(u,z)$, $d$ are given by \eqref{ch var} and \eqref{d}, respectively and $u_0, z_0$ are the following
\begin{equation*}
	\begin{cases}
		u_0(x)=\Pi v_0=v_0(x)-\int_{0}^{x}k(x,y)v_0(y) dy,\\
		z_0(x)=\Pi w_0=w_0(x)-\int_{0}^{x}k(x,y)w_0(y) dy.
	\end{cases}
\end{equation*}
In order to establish the exponential stabilizability of the system \eqref{FHNlin ET}, we first prove that the target system \eqref{FHNtargetvariable} is exponentially stable and then, thanks to the invertibility of the Volterra transformation \eqref{volterra}, we get the stabilization result for \eqref{FHNlin ET}. It is worth mentioning that, in the continuous backstepping case instead of the target system \eqref{FHNtargetvariable} with nonhomogeneous boundary data at the right Dirichlet end, we have a similar system with homogeneous boundary data, and thus one can obtain the exponential stability with critical exponential decay rate up to $e^{-\delta t}$, see \cite{chowdhury2024local} for more details. For the event-triggering case, due to the presence of a nonhomogeneous boundary data $u(t,1)=d(t)$ in the target system \eqref{FHNtargetvariable}, we get the exponential stability of the system \eqref{FHNtargetvariable} with a decay rate $e^{-\omega t}$ when $\omega<\delta.$ This fact essentially implies the same result for the original event-triggered system \eqref{FHNlin ET}.

\section{Well-posedness result}\label{wellp}

Let $\{t_j\}_{j\in J}$, $J\subset \mathbb N$, be an increasing sequence of times. We start by introducing the maximal time $T$ under which the system \eqref{FHNlin ET} possess a solution
\begin{equation}
T:=\begin{cases}\infty, & \text{ if } \{t_j\} \text{ is a finite sequence },\\
\limsup\limits_{j\to \infty} t_j & \text{ otherwise}. 
\end{cases}
\end{equation}
\begin{theorem}\label{wellposed 1}
	For every $\left(v(t_j), w(t_j)\right)\in L^2(0,1)\times L^2(0,1)$, system \eqref{FHNlin ET} possess a unique solution $(v,w)\in \mc C^0([t_j, t_{j+1}]; L^2(0,1)\times L^2(0,1))\cap L^2(t_j, t_{j+1}; H^1(0,1)\times L^2(0,1)).
 $ 
\end{theorem}
\begin{proof}
We prove this result by establishing a similar result for a system with homogeneous boundary conditions and then performing a change of variable, we conclude the result for the main system \eqref{FHNlin ET}. Note that the boundary data for the right end (at $x=1)$ of $v$ is constant in the interval $(t_j, t_{j+1})$ and we denote $\bar q=q_d.$ Next, we consider the following system
\begin{equation}\label{FHNlin ET1}
	\begin{cases}
		\pa_t \bar v=\pa_{x}^2\bar v-a\bar v-\rho \bar w-\bar q\left(a+\frac{\rho \gamma}{\delta}\right)x &\text{ in } (t_j, t_{j+1}) \times (0,1),\\
		\pa_t \bar w=\gamma \bar v-\delta \bar w &\text{ in }  (t_j, t_{j+1}) \times (0,1),\\
		\bar v(t,0)=0,\,\, \bar v(t,1)=0& \text{ in } (t_j, t_{j+1}),
	\end{cases}.
\end{equation}
As it is given that $\left(v(t_j), w(t_j)\right)\in L^2(0,1)\times L^2(0,1)$, thanks to \cite[Lemma A.1]{chowdhury2024local}, the solution of \eqref{FHNlin ET1} lies in the space 
$\mc C^0([t_j, t_{j+1}]; L^2(0,1)\times L^2(0,1))\cap L^2(t_j, t_{j+1}; H^1(0,1)\times L^2(0,1))$. 
Next taking, $v=\bar v+\bar q x, w=\bar w+\frac{\bar q \gamma}{\delta}x,$ $ \bar q=\int_{0}^{1}k(1,y)v(t_j,y) dy,$ one can show that $(v,w)$ satisfies \eqref{FHNlin ET} and $$(v,w)\in \mc C^0([t_j, t_{j+1}]; L^2(0,1)\times L^2(0,1))\cap L^2(t_j, t_{j+1}; H^1(0,1)\times L^2(0,1)).$$ 
\end{proof}
The above well-posedness result allows us to construct a solution $(v,w)$ for the system \eqref{FHNlin ET} in the time interval $[0, T).$
\begin{cor}\label{cor wellp}
	Let $(v_0, w_0)\in L^2(0,1)\times L^2(0,1).$ Then there exists unique mapping $(v,w)$ such that $(v,w)\in \mc C^0\left([0, T); L^2(0,1)\times L^2(0,1)\right)$ with $v(t)\in  H^1(0,1),$ for a.e. $t\in \left(0, T\right)$ and $(v,w)$ satisfies \eqref{FHNlin ET} in $(t_j, t_{j+1}), \forall j\geq 0.$ 
\end{cor}
\begin{proof}
	First let us consider that $(v_0, w_0)\in L^2(0,1)\times L^2(0,1).$ Then using \Cref{wellposed 1}, we get the existence of unique solution of \eqref{FHNlin ET} in $(0,t_1)$ denoted by $(v^0,w^0)$ and we have  $(v^0,w^0) \in \mc C^0([0, t_{1}]; L^2(0,1)\times L^2(0,1)).$ with $v^0(t) \in  H^1(0,1)$ for a.e. $t\in (0, t_{1}].$ 
 Note that we get $(v^0(t_1), w^0(t_1))\in L^2(0,1)\times L^2(0,1).$ Thus we start the system in the next interval $(t_1, t_2)$ with taking $(v^0(t_1), w^0(t_1))$ as the initial data and so on. That is, at each step, we have the following system for $(j\geq 1)$
	\begin{equation}\label{FHN wellp}
		\begin{cases}
			\pa_t v^j=\pa_{x}^2v^j-av^j-\rho w^j&\text{ in } (t_j, t_{j+1}) \times (0,1),\\
		\pa_t	w^j=\gamma v^j-\delta w^j &\text{ in }  (t_j, t_{j+1}) \times (0,1),\\
			v^j(t,0)=0,\,\, v^j(t,1)=\int_{0}^{1}k(1,y)v^j(t_j,y) dy& \text{ in } (t_j, t_{j+1}),\\
			v^j(t_j,x)=v^{j-1}(t_j,x), \, 	w^j(t_j,x)=w^{j-1}(t_j,x)&  \text{ in } (0,1).
		\end{cases}
	\end{equation}
and we get the existence of unique solution $(v^j, w^j)\in \mc C^0([t_j, t_{j+1}]; L^2(0,1)\times L^2(0,1))$ with $\left(v^j(t),w^j(t)\right)\in  H^1(0,1)\times L^2(0,1),$ for a.e $t\in (t_j, t_{j+1}]. $ 
With this step-by-step construction, we finally have the solution $(v,w)=(v^j, w^j)$ in $[t_j, t_{j+1}]$ with the desired regularity.
\end{proof}
\section{Exponential stabilizability}\label{exp stab}
The main goal of this section is to present the event-triggered boundary control and its primary outcomes: avoidance of the Zeno phenomenon and the exponential stability of the event-triggered controlled system. The event-triggered boundary control discussed in this paper includes a triggering condition, which determines the moments when the controller needs to be updated, and a backstepping-based boundary feedback law, applied at the right end of the Dirichlet boundary for the parabolic component. The proposed event-triggering condition is based on the evolution of the $L^2$ norms of the states of both the parabolic and the ODE components.
\subsection{Event-triggered rule}
Let us denote $V(t)=\norm{v(t)}+\norm{w(t)},$ where $(v,w)$ is the solution of the equation \eqref{FHNlin ET}.
\begin{defn}\label{defn-ev}
Let $\beta>0$ be a parameter that will be chosen later. Let us recall the kernel $k$ defined in \eqref{kernel} and the control $d$ in \eqref{d}. Let us define the following set:
\begin{equation}\label{tr rule}
E(t_j):=\{t>0: t>t_j \text{ and } |d(t)|> \beta \norm{k(1,\cdot)}V(t)+\beta \norm{k(1,\cdot)}V(t_j)\}.
\end{equation}
The event-triggered boundary control is defined by considering the following  components: 

\begin{enumerate}
	\item (The event-trigger) The times of the events $t_j\geq 0$  with $t_0=0$  form a finite or countable set of times  which is determined by the following rules for some $j\geq0$: \\
	\begin{itemize}
		\item [a)] if 
		$E(t_j) = \emptyset$  then the set of the times of the events is $\{t_{0},...,t_{j}\}$.\\ 
		\item [b)] if $ E(t_j)
		\neq \emptyset$, then the next event time is given by: \begin{equation}\label{triggering}t_{j+1}=\inf E(t_j)
		\end{equation}
	\end{itemize}
	
	\item (The control action)  The boundary feedback law,
	\begin{equation}\label{controlfunction}
		q_d(t) =    \int_{0}^{1} k(1,y)v(t_{j},y) dy, \quad  \forall t \in [t_{j},t_{j+1}).
	\end{equation}
\end{enumerate}
\end{defn}
\subsection{Avoidance of Zeno behavior}
In this section, we will prove that the so-called Zeno phenomenon, which refers to an infinite number of triggers in a finite time interval, is avoided by ensuring minimal dwell time between two triggering instances. Let us first present the following intermediate result before moving to the finding of the existence of the minimal dwell time.
\begin{lem}\label{Estimate_of_supNorm}
	For the closed-loop system \eqref{FHNlin ET}, the following estimate holds, for all $t \in [t_{j},t_{j+1}]$, $j\geq 0$:
	\begin{equation}\label{upper_bound_supNormofU}
		\sup_{t_{j} \leq s \leq t_{j+1}}(\Vert v(s) \Vert + \Vert w(s) \Vert ) \leq M_j (\Vert v(t_j) \Vert + \Vert w(t_j) \Vert)
	\end{equation}
	where $M_j= \sqrt{2}e^{\frac{c}{2}(t_{j+1}- t_{j})}\left(1+   \Vert k(1,\cdot) \Vert + \frac{\Vert k(1,\cdot) \Vert}{\sqrt{c}}\right) +   \Vert k(1, \cdot) \Vert$ and $c$ is a positive constant depends on the system parameters.  
\end{lem}
\begin{proof}
Let us first introduce the following change of variables
	\begin{equation}\label{change_of_variable}
		\begin{cases}
		\tl v(t,x) = v(t,x) - x q_d(t), \\
		\tl w(t,x)=w(t,x).
		\end{cases}
	\end{equation}
	It is easy to check that $(\tl v, \tl w)$ satisfies the following PDE for all $t \in (t_{j},t_{j+1})$, $j\geq 0$,
	\begin{equation}\label{sysparabolic}
		\begin{cases}
		\tl v_t =  \pa_{x}^2 \tl v- a \tl v -\rho \tl w -a  x q_d &\text{ in }  (t_j, t_{j+1}) \times (0,1), \\
		\tl w_t=\gamma \tl v-\delta \tl w+\gamma x q_d(t) &\text{ in }  (t_j, t_{j+1}) \times (0,1),\\
		\tl v(t,0)=0,\,\, \tl v(t,1)=0 & \text{ in } (t_j, t_{j+1}).
		\end{cases}
	\end{equation} 
	 Next, by considering the  function $W(t)=\frac{1}{2}\left(\Vert \tl v(t) \Vert^2+\Vert \tl w(t) \Vert^2\right)$    
	and taking its time derivative along the solutions of  \eqref{sysparabolic} and using Cauchy-Schwarz inequality for the coupling terms, we obtain, for $t\in (t_{j}, t_{j+1})$:
	\begin{align*}
		 \dot{W}(t) &\leq -\Vert \pa_x \tl v(t) \Vert^2-a\Vert \tl v(t) \Vert^2-\delta\Vert \tl w(t)\Vert^2 +{(\rho+\gamma)}W(t)+q_d\int_{0}^{1}(\gamma x \tl w-ax\tl v) dx.
		 \end{align*}  
	In addition, using the Young's inequality on the last term along with the Cauchy-Schwarz inequality, we get 
	\begin{equation*}
		\dot{W}(t) \leq c_1 W(t)  + \frac{1}{2}q_d^2 + c_2 W(t),
	\end{equation*} where $c_1, c_2$ are two positive constants.
	Then, for $t \in (t_{j}, t_{j+1})$:
	\begin{equation*}
		\dot{W}(t) \leq c W(t) + \frac{1}{2}q_d^2,
	\end{equation*}
	where $c$ is a positive constant. Thanks to the Gronwall's inequality on an interval $[p,r]$ where $p > t_{j}$ and $r < t_{j+1}$, one gets, for all $t \in [p,r]$:
	\begin{equation*}
		W(t) \leq e^{c(t-p)}(W(p)+ \frac{1}{ 2c}q_d^2).
	\end{equation*}
	Due to the continuity of $W(t)$ on $[t_{j},t_{j+1}]$  and the fact that $p,r$ are arbitrary, we can conclude that
\begin{equation}\label{Estimate_functional_p_tk+1}
		W(t) \leq e^{c(t_{j+1}-t_{j})}\left(W(t_{j})+ \tfrac{1}{2 c}q_d^2\right)
	\end{equation}
	for all $t \in [t_{j},t_{j+1}]$. Applying the Cauchy-Schwarz inequality, we have that $\vert q_d \vert \leq \Vert k(1,\cdot) \Vert \Vert v(t_j) \Vert$. Using this fact in  \eqref{Estimate_functional_p_tk+1}, we get, in addition:
	\begin{equation*}
		\Vert \tl v(t) \Vert^2 + \Vert \tl w(t) \Vert^2\leq e^{c(t_{j+1}-t_{j})} \left(\Vert\tl v(t_j) \Vert^2 + \Vert\tl w(t_j) \Vert^2+\frac{1}{c} \Vert k(1,\cdot) \Vert^2 \Vert v(t_j) \Vert^2 \right),
	\end{equation*} 
which further gives:
 \begin{equation}\label{est1}
 	\Vert \tl v(t) \Vert + \Vert \tl w(t) \Vert\leq\sqrt{2} e^{\frac{c}{2}(t_{j+1}-t_{j})} \left(\Vert\tl v(t_j) \Vert + \Vert\tl w(t_j) \Vert+\frac{1}{\sqrt c} \Vert k(1,\cdot) \Vert \Vert v(t_j) \Vert \right).
 \end{equation}
thanks to the change of variables \eqref{change_of_variable} and the triangle inequalities, we obtain the following inequalities:  
	\begin{equation*}
		\begin{split}
			\Vert  v(t) \Vert \leq \Vert \tl v(t)\Vert +   \vert q_d\vert, \\
			\Vert \tl v(t_j) \Vert \leq \Vert  v(t_j)\Vert +   \vert q_d\vert,
		\end{split}
	\end{equation*}
	together with $\vert q_d \vert \leq \Vert k(1,\cdot) \Vert \Vert v(t_j) \Vert$. Thanks to \eqref{est1} and the above estimates, we obtain, for all $t \in [t_{j}, t_{j+1}]$,
	\begin{align*}
	(\Vert v(t) \Vert +\Vert w(t) \Vert)&\leq (\Vert \tl v(t) \Vert +\Vert k(1,\cdot) \Vert \Vert v(t_j) \Vert+\Vert \tl w(t) \Vert)\\
	&\leq \sqrt{2} e^{\frac{c}{2}(t_{j+1}-t_{j})} \left(\Vert\tl v(t_j) \Vert + \Vert\tl w(t_j) \Vert+\frac{1}{\sqrt c} \Vert k \Vert \Vert v(t_j) \Vert \right)+\Vert k(1,\cdot) \Vert \Vert v(t_j) \Vert\\
	& \leq \sqrt{2} e^{\frac{c}{2}(t_{j+1}-t_{j})} \left(\Vert v(t_j) \Vert +\Vert k(1,\cdot) \Vert \Vert v(t_j) \Vert+ \Vert w(t_j) \Vert+\frac{1}{\sqrt c} \Vert k(1,\cdot) \Vert \Vert v(t_j) \Vert \right)\\
 &\quad +\Vert k(1,\cdot) \Vert \Vert v(t_j) \Vert.
	\end{align*}
Therefore, finally, we deduce the following:
	\begin{equation*}
		\sup_{t_{j} \leq s \leq t_{j+1}}(\Vert v(s) \Vert +\Vert w(s) \Vert) \leq M_j (\Vert v(t_j) \Vert+\Vert w(t_j) \Vert)
	\end{equation*}
	with $M_j= \sqrt{2}e^{\frac{c}{2}(t_{j+1}- t_{j})}\left(1+   \Vert k(1,\cdot) \Vert + \frac{\Vert k(1,\cdot) \Vert}{\sqrt{c}}\right) +   \Vert k(1, \cdot) \Vert$.  This concludes the proof.
\end{proof}
\begin{theorem}\label{minimal_time}
	Under the event-triggered boundary control \eqref{tr rule}-\eqref{triggering}-\eqref{controlfunction},  there exists a minimal dwell-time between two triggering times, i.e. there exists a constant $\tau >0$ (independent of the initial condition $(v_0,w_0)$) such that $t_{j+1} - t_{j} \geq \tau $,  for all $j \geq 0$.
\end{theorem}
\begin{proof} 
	Let us consider the following function $g \in C^2([0,1])$  :
	\begin{equation}\label{definition_g}
		g(x):= \sum_{n=1}^{N}k_n \phi_{n}(x)
	\end{equation}  
	where $N \in \N$, $k_n := \int_{0}^{1}k(1, y)\phi_n(y)dy$,  with $k$ satisfying \eqref{kernel} and  $\phi_n(x) =\sqrt{2}\sin(n \pi x) ,n=1,2...$ 
	Let us recall $(v,w)$ is the solution of the event-triggered control system \eqref{FHNlin ET}.  Next, we define
	\begin{equation}\label{tilde_deviation}
		\tilde{d}(t)= \int_{0}^{1} g(y)\left(v(t_{j},y) - v(t,y)\right)dy
	\end{equation}
	for $t \in [t_{j},t_{j+1})$, for $j\geq 0$ and $g$ is given  by \eqref{definition_g}. Taking the time derivative of  $\tilde{d}(t)$ along the solutions of \eqref{FHNlin ET} yields, for all $t \in [t_{j},t_{j+1})$:
	\begin{equation*}
		\begin{split}
			\dot{\tilde{d}}(t) =&  \left( g'(1)v(t,1)- g'(0)v(t,0)  \right)
			 +  \left(g(0)\partial_x v(t,0)  -g(1)\partial_x v(t,1))\right) \\
			&+ \int_{0}^{1}\left(-g''(y)+ag(y)\right)v(t,y)dy+ \int_{0}^{1}\rho g(y)w(t,y)dy
		\end{split}
	\end{equation*}
	Note that  $g(1)\partial_x v(t,1) =g(0)\partial_x v(t,0)= 0,$  since the function  $g$ evaluated at $x=0,1$ vanishes, and $v(t,0)=0$. 
	In addition, by  the expression of $g$, we have   
	\begin{equation*}
		\begin{split}
			\dot{\tilde{d}}(t) =&  \int_{0}^{1}k(1,y)v(t_{j},y)dy \sum_{n=1}^{N}k_n \phi_n'(1) \\
			& + \sum_{n=1}^{N}k_n (n^2\pi^2+a) \int_{0}^{1}\phi_n(y)v(t,y)dy+\rho \sum_{n=1}^{N}k_n  \int_{0}^{1}\phi_n(y)w(t,y)dy
		\end{split}
	\end{equation*}
	Using the Cauchy-Schwarz inequality and $\Vert \phi_n \Vert=1$ for $n=1,2,...$ the following estimate holds for $t \in (t_{j},t_{j+1})$, $j\geq 0$:
	\begin{equation}\label{tilded2}
		\vert \dot{\tilde{d}}(t) \vert   \leq   \Vert k(1,\cdot) \Vert \Vert v(t_j) \Vert F_N + \left(\Vert v(t) \Vert+\Vert w(t) \Vert\right) G_{N}
	\end{equation}
	where   $F_N := \sum_{n=1}^{N}\Big\vert k_n \phi_n'(1)  \Big \vert $ and  $G_N := \sum_{n=1}^{N} \vert k_n (n^2\pi^2+a+\rho) \vert $.
	Therefore, from \eqref{tilded2} along with the fact  $\tilde{d}(t_{j})=0$, we obtain the following estimate:
	\begin{equation}\label{tilded3}
		\vert \tilde{d}(t) \vert \leq (t -t_{j}) \Vert k(1,\cdot) \Vert  V(t_j)  F_N + (t-t_{j})\sup_{t_{j} \leq s \leq t}( V(s) ) G_N
	\end{equation}
	Note that using \eqref{d} and \eqref{tilde_deviation}, the deviation $d(t)$ can be expressed as follows:
	\begin{equation}\label{d(t)tilde}
		d(t) = \tilde{d}(t) + \int_{0}^{1}(k(1,y)-g(y))(v(t_{j},y) - v(t,y))dy. 
	\end{equation} 
	Hence,  combining \eqref{tilded3} and \eqref{d(t)tilde}, we obtain an estimate of $d$ as follows:  
	\begin{equation}\label{normk-g}
		\begin{split}
			\vert d(t) \vert & \leq (t -t_{j}) \Vert k(1,\cdot) \Vert  V(t_j)  F_N + (t-t_{j})\sup_{t_{j} \leq s \leq t}( V(s) ) G_N  \\
			& \hskip 2 cm+ \Vert k(1,\cdot)-g \Vert\Vert v(t_j) \Vert + \Vert k(1,\cdot)-g \Vert\Vert v(t) \Vert.
		\end{split}
	\end{equation}
	 Let us first assume that $V(t_j) \neq 0$. Using \eqref{normk-g} and assuming that an event is triggered  at $t=t_{j+1}$, we have
	\begin{equation}\label{at_tk+1}
		\begin{split}
			\vert d(t_{j+1}) \vert & \leq (t_{j+1} -t_{j}) \Vert k(1,\cdot) \Vert  V(t_j)  F_N \\
			& \hskip 2 cm + (t_{j+1}-t_{j})\sup_{t_{j} \leq s \leq t_{j+1}}(\Vert V(s) \Vert) G_N  \\
			& \hskip 2 cm+ \Vert k(1,\cdot)-g \Vert V(t_j)  + \Vert k(1,\cdot)-g \Vert V(t_{j+1}) 
		\end{split}
	\end{equation}
	and, by \Cref{defn-ev},  we have that, at $t=t_{j+1}$
	\begin{equation}\label{d_at_tk+1}
		\vert d(t_{j+1}) \vert \geq  \beta \Vert k(1,\cdot) \Vert  V(t_j)  + \beta \Vert k(1, \cdot) \Vert  V(t_{j+1}).  	
	\end{equation}
Combining \eqref{at_tk+1} and \eqref{d_at_tk+1}, we get
	\begin{equation*}
		\begin{split}
			&\beta\Vert k(1, \cdot) \Vert \Vert V(t_j) \Vert  + \beta \Vert k(1, \cdot) \Vert \Vert V(t_{j+1}) \Vert   \\
			&  \leq (t_{j+1} -t_{j}) \Vert k(1, \cdot) \Vert F_N  V(t_j)   + (t_{j+1}-t_{j})G_N  \sup_{t_{j} \leq s \leq t_{j+1}}( V(s) )  \\
			&  \hskip 2cm + \Vert k(1,\cdot)-g \Vert V(t_j)  + \Vert k(1)-g \Vert V(t_{j+1}).
		\end{split}
	\end{equation*}
	and hence,
	\begin{equation*}
		\begin{split}
			& \left(\beta\Vert k(1,\cdot) \Vert  -  \Vert k(1,\cdot)-g \Vert  \right)  V(t_{j+1})  +  \left( \beta \Vert k(1,\cdot) \Vert  -  \Vert k(1,\cdot)-g \Vert  \right)  V(t_j)    \\
			&  \leq (t_{j+1} -t_{j}) \Vert k(1,\cdot) \Vert F_N  V(t_j)    +   (t_{j+1}-t_{j}) \, G_N \sup_{t_{j} \leq s \leq t_{j+1}}( V(s) ).
		\end{split}
	\end{equation*}
	By the definition of $g$, it is obvious that $\lim\limits_{N\to \infty} g(y)=k(1,y)$ as $g$ is simply the $N$-mode truncation of the Fourier series expansion of $k(1,\cdot)$. We select $N \geq 1$ in \eqref{definition_g} sufficiently large so  that $\Vert k(1,\cdot) -g \Vert < \beta\Vert k(1,\cdot) \Vert$. 
	Notice that we can always find a $N$ sufficiently large so that the condition $\Vert k(1,\cdot) -g \Vert < \beta\Vert k(1,\cdot) \Vert$ holds, since $\Vert k(1,\cdot)-g\Vert$ tends to zero as $N$ tends to infinity.
	In addition, using the fact that $\Vert V(t_{j+1}) \Vert \geq 0$ and  by \eqref{upper_bound_supNormofU} in  \Cref{Estimate_of_supNorm}, we obtain the following estimate:
	\begin{equation*}
		\begin{split}
			& \left( \beta\Vert k(1,\cdot) \Vert  -  \Vert k(1,\cdot)-g \Vert  \right)  V(t_j)    \\
			&  \leq (t_{j+1} -t_{j}) \Vert k(1,\cdot) \Vert F_N  V(t_j)  + (t_{j+1}-t_{j})G_N M_j  V(t_j)   
		\end{split}
	\end{equation*} 
	where $M_j = \sqrt{2}e^{c/2(t_{j+1}- t_{j})}\left(1+   \Vert k(1,\cdot) \Vert + \frac{\Vert k(1,\cdot) \Vert}{\sqrt{c}}\right) +   \Vert k(1, \cdot) \Vert$. 
\begin{itemize}
		\item $a_0:= \beta\Vert k(1,\cdot) \Vert  -  \Vert k(1,\cdot)-g \Vert $
		\item $a_1:=  \Vert k(1,\cdot) \Vert F_N + G_N   \Vert k(1,\cdot) \Vert$
		\item $a_2:=\sqrt{2} G_N \left(1 +  \Vert k(1,\cdot) \Vert + \tfrac{\Vert k(1,\cdot) \Vert}{\sqrt{c}}\right)$
	\end{itemize}
	we obtain an inequality of the form:
	\begin{equation}\label{dwell_time}
	0<	a_0 \leq a_1 (t_{j+1} - t_{j}) + a_2 (t_{j+1} - t_{j}) e^{\frac{c}{2} (t_{j+1} - t_{j})}
	\end{equation}
	from which we aim at finding a lower bound for $(t_{j+1} - t_{j})$. 
	 Let us denote it as   $\alpha(s):= a_1 s + a_2 s e^{\frac{c}{2} s}  $ with $s= (t_{j+1} - t_{j})$. Since $a_0$ is strictly positive, then there exists $\tau >0$ such that $s= (t_{j+1} - t_{j})\geq \tau > 0$.
  Indeed, let us write $\tau=\inf \mc A, \text{ where } \mc A=\{s\in (0, \infty)| \alpha(s)\geq a_0>0\}.$ Clearly $\tau\geq 0.$ If $\tau=0,$ there exists a sequence $\{s_n\}\in \mc A$ of nonnegative real numbers such that $s_n \to 0.$ As $\alpha$ is continuous function $\alpha(s_n)\to 0,$ which is a contradiction to the fact $\alpha(s_n)\geq a_0>0,$ as $s_n \in \mc A.$
  
	If $V(t_j) = 0$, then \Cref{Estimate_of_supNorm} guarantees that $V(t)$ remains zero. In this case, by  \Cref{defn-ev},  we do not need to employ triggering anymore, and thus, the Zeno phenomenon is immediately excluded. This concludes the proof.
\end{proof} 
Thus, we already proved that there is a minimal dwell time (which is uniform and does not depend on either the initial condition or on the state of the system), and no Zeno solution can appear. Moreover, one can recover the explicit minimal dwell time numerically using the Lambert W function, see \cite[Section 3.1.1]{MK21} for more details.
Henceforth, we have the following result on the existence of solutions of the system \eqref{FHNlin ET} with \eqref{tr rule}-\eqref{triggering}-\eqref{controlfunction} for all $ t>0$ which is essentially the combination of \Cref{cor wellp} and \Cref{minimal_time}.
\begin{cor}\label{wellposed-main}
	Let $(v_0, w_0)\in L^2(0,1)\times L^2(0,1).$ Then there exists a unique mapping $(v,w)$ such that $(v,w)\in \mc C^0\left([0, \infty); L^2(0,1)\times L^2(0,1)\right)$ with $v\in  L^{2}((t_j, t_{j+1}); H^1(0,1)), $ for all $j\geq $ and $(v,w)$ satisfies \eqref{FHNlin ET} for $t\in (t_j, t_{j+1})$, $\forall j\geq 0.$ 
\end{cor}
\subsection{Stability analysis}
This section is devoted to the main result regarding the exponential stabilization of the linearized FHN system \eqref{FHNlin ET}. We start by showing that the target system \eqref{FHNtargetvariable} is Input-to-state stable with respect to $d(t),$ defined in \eqref{d}.  
\begin{prop}[Input-to-state stability]\label{lmm iss}
	Let $\epsilon\in (0,\delta)$ be given. Then the target system \eqref{FHNtargetvariable} 
 satisfies the following estimate:
	\begin{equation*}
 \norm{u(t)}+\norm{z(t)}\leq C e^{-\delta t}\left(\norm{u(0)}+\norm{z(0)}\right)+\vartheta \sup_{0\leq s\leq t}\left(e^{-(\delta-\epsilon) (t-s)} |d(s)|\right),
		\end{equation*}
	where $C>0$ is uniform with respect to $\epsilon$ and \begin{equation}\label{vartheta}\vartheta =\frac{2\pi^2}{{(\lambda+\pi^2-\delta)^2}-4\rho \gamma}\left(1+ \bigg[\frac{1}{\epsilon}+\frac{1}{\pi^2+\lambda-\delta+\epsilon} \bigg]\right).\end{equation}
	\end{prop}
\begin{proof}
The continuity of the transformation \eqref{volterra} and the above well-posedness result \Cref{wellposed-main} for the main control system \eqref{FHNlin ET} imply that the system \eqref{FHNtargetvariable} has a unique solution $(u,z)$ such that $(u,z)\in \mc C^0\left([0, \infty); L^2(0,1)\times L^2(0,1)\right)$ with $u(t)\in  H^1(0,1), $ for a.e. $t\in \left(t_j,t_{j+1}\right)$. 

\noindent
Let us denote \begin{align*}
u_n(t)=\sqrt{2}\int_{0}^{1}u(t,x)\sin(n\pi x) dx, \quad  z_n(t)=\sqrt{2}\int_{0}^{1}z(t,x)\sin(n\pi x) dx.
\end{align*}
	As $\{\sqrt{2}\sin(n\pi x)\}_{n\in\mathbb N}$ forms an orthonormal basis of $L^2(0,1)$, using Parseval's identity we have that
	\begin{equation*}
	\norm{u(t)}^2=\sum_{n=1}^{\infty}|u_n(t)|^2, \, \norm{z(t)}^2=\sum_{n=1}^{\infty}|z_n(t)|^2.
	\end{equation*}
	Then $(u_n, z_n)$ satisfy the following equation:
	\begin{equation*}
		\begin{cases}
	\dot{u}_n(t)+(n^2\pi^2+\lambda)u_n(t)+\rho z_n(t)=(-1)^n n\pi d(t),\\
	\dot{z}_n(t)+\delta z_n(t)=\gamma  u_n(t).
	\end{cases}
	\end{equation*}
Let us write the above system in the following abstract form
\begin{equation}\label{ST1}
	\dot{\mathbf{U}}_n(t)=A_n\mathbf{U}_n(t)+F_n(t), \mathbf{U}_n(t)=(u_n(t), z_n(t))^T, \quad  \mathbf{U}_n(0)=(u_{n,0}, z_{0,n})^T. 
\end{equation}
where $$
A_n=
\begin{bmatrix}
-(n^2\pi^2+\lambda)	 & -\rho  \\
	\gamma & -\delta
\end{bmatrix}, \quad F_n(t)= ((-1)^n n\pi d(t),0)^T$$
 We decompose $A_n$ as
$$A_n=\begin{bmatrix}
	-\delta&0\\0&-\delta
\end{bmatrix} 
+\begin{bmatrix}
	-(n^2\pi^2+\lambda-\delta)&-\rho\\\gamma&0
\end{bmatrix}=A_1+B_n.$$
 Since the matrix $A_1$ and $B_n$ commute, we have
$e^{A_n t}=e^{A_1t}e^{B_nt}$.
We choose the damping parameter $\lambda$ large enough. In particular, we can set $\lambda>\delta$ and also $\sqrt{{(\lambda+n^2\pi^2-\delta)^2}-4\rho \gamma}>0$. 
For each $n\in \N$, eigenvalues of the matrix $B_n$ are given by
\begin{align*}
&\lambda_n=\frac{1}{2}\bigg[-({\lambda+n^2\pi^2-\delta})+\left(\sqrt{{(\lambda+n^2\pi^2-\delta)^2}-4\rho \gamma}\right)\bigg],\\ 
&\mu_n=\frac{1}{2}\bigg[-({\lambda+n^2\pi^2-\delta})-\left(\sqrt{{(\lambda+n^2\pi^2-\delta)^2}-4\rho \gamma}\right)\bigg].
\end{align*}
It can be easily checked that the eigenvalues of the matrix $B_n$ are negative for all $n \in \N.$ Furthermore, for large values of $n$, we have the following behavior of the eigenvalues of $B_n$
\begin{equation*}\lambda_n\approx -\frac{\rho \gamma}{n^2\pi^2}+O\left(\frac{1}{n^4}\right), \quad \mu_n\approx -(n^2\pi^2+\lambda-\delta)+O\left(\frac{1}{n^2}\right).
\end{equation*}
Also we have \begin{align*}
	(\lambda_n-\mu_n)=\sqrt{{(\lambda+n^2\pi^2-\delta)^2}-4\rho \gamma}>0 \,\, (\text{ by the choice of } \lambda)
\end{align*}
and the following asymptotic expression  \begin{equation*}
(\lambda_n-\mu_n)\approx n^2\pi^2+\lambda-\delta-\frac{2\rho \gamma}{n^2\pi^2}+O\left(\frac{1}{n^4}\right).
\end{equation*}
A direct computation shows that,$$e^{B_nt}=\frac{1}{\lambda_n-\mu_n}\begin{bmatrix}
	\lambda_ne^{\lambda_nt}-\mu_ne^{\mu_n t}& \frac{\lambda_n\mu_n}{\gamma}\left(e^{\mu_nt}-e^{\lambda_nt}\right)\\ \\
	\gamma(e^{\lambda_nt}-e^{\mu_nt})	&\lambda_ne^{\mu_nt}-\mu_ne^{\lambda_nt}
\end{bmatrix}$$ for all $n \in \N$. Therefore,
$$e^{A_n t}=\frac{1}{\lambda_n-\mu_n}e^{-\delta t}\begin{bmatrix}
\lambda_ne^{\lambda_nt}-\mu_ne^{\mu_n t}& \frac{\lambda_n\mu_n}{\gamma}\left(e^{\mu_nt}-e^{\lambda_nt}\right)\\ \\
 \gamma(e^{\lambda_nt}-e^{\mu_nt})	&\lambda_ne^{\mu_nt}-\mu_ne^{\lambda_nt}
\end{bmatrix}.$$\\
The corresponding solution $\mathbf{U}_n(t)$ of \eqref{ST1} is given by
$$\mathbf{U}_n(t)=e^{A_nt}\mathbf{U}_0+\int_{0}^{t}e^{A_n(t-s)}F_n(s) ds.$$
Thus we have the following expressions for the solution of \eqref{ST1}:
\begin{align*}
\nonumber&u_n(t)=\frac{e^{-\delta t}}{\lambda_n-\mu_n}\bigg[u_{0,n}(	\lambda_ne^{\lambda_nt}-\mu_ne^{\mu_n t})+\frac{\lambda_n\mu_n}{\gamma}z_{0,n}(e^{\mu_n t}-e^{\lambda_n t})\bigg]\\
& \hspace{5cm}+\frac{(-1)^n n\pi}{\mu_n-\lambda_n}\int_{0}^{t}e^{-\delta (t-s)}\left(\lambda_ne^{\lambda_n (t-s)}- \mu_ne^{\mu_n (t-s)}\right)d(s) ds\\
\nonumber&z_n(t)=\frac{e^{-\delta t}}{\lambda_n-\mu_n}\bigg[u_{0,n}\gamma(e^{\lambda_n t}- e^{\mu_n t})+z_{0,n}(\lambda_n e^{\mu_n t}-\mu_ne^{\lambda_n t})\bigg]\\
&\hspace{5cm}+\frac{(-1)^n n\pi\gamma}{\lambda_n-\mu_n}\int_{0}^{t}e^{-\delta (t-s)}\left( e^{\lambda_n (t-s)}-e^{\mu_n (t-s)}\right)d(s) ds.
\end{align*}
In order to find the stability estimate for the system \eqref{FHNtargetvariable}, it is enough to establish the same for $(u_n, z_n)$, the solution of \eqref{ST1}.

\noindent
\paragraph{\underline{\it{Estimate of $u_n(t)$}}}
\begin{align}\label{un}
\nonumber&	|{u_n(t)}|^2\leq \frac{e^{-2\delta t}}{|\mu_n-\lambda_n|^2}\bigg[|u_{0,n}|^2(\lambda_ne^{\lambda_n t}-\mu_n e^{\mu_n t})^2+|z_{0,n}|^2\rho^2\gamma(e^{\lambda_n t}- e^{\mu_n t})^2\bigg]\\
\nonumber	& \hspace{5cm}+\frac{ n^2\pi^2}{|\mu_n-\lambda_n|^2}{\left(\int_{0}^{t}e^{-\delta (t-s)}\left(\lambda_n e^{\lambda_n (t-s)}- \mu_n e^{\mu_n (t-s)}\right)d(s) ds\right)^2}\\
	&\leq C {e^{-2\delta t}}\bigg[|u_{0,n}|^2 \frac{(|\lambda_n|^2+|\mu_n|^2)}{|\mu_n-\lambda_n|^2}+|z_{0,n}|^2\frac{2}{|\mu_n-\lambda_n|^2}\bigg]+\frac{ n^2\pi^2}{|\mu_n-\lambda_n|^2} \mc I^2.
\end{align}
Let us estimate the term $\mc I=\int_{0}^{t}e^{-\delta (t-s)}\left(\lambda_ne^{\lambda_n (t-s)}- \mu_ne^{\mu_n (t-s)}\right)d(s) ds$.
\begin{align}\label{i}
	\nonumber	|\mc I|&\leq \sup_{0\leq s\leq t}\left(e^{-\delta (t-s)} |d(s)|\right)\int_{0}^{t}\left|\left(\lambda_ne^{\lambda_n (t-s)}-\mu_ne^{\mu_n (t-s)}\right)\right|ds\\
	\nonumber	&\leq \sup_{0\leq s\leq t}\left(e^{-\delta (t-s)}|d(s)|\right)\left(|\lambda_n|e^{\lambda_n t}\int_{0}^{t}e^{-\lambda_ns}+|\mu_n|e^{\mu_nt}\int_{0}^{t}e^{-\mu_ns}\right)\\
	\nonumber	&\leq \sup_{0\leq s\leq t}\left(e^{-\delta (t-s)}|d(s)|\right)\left(\frac{|\lambda_n|}{\lambda_n}(e^{\lambda_nt}-1)+\frac{|\mu_n|}{\mu_n}(e^{\mu_nt}-1)
	\right)\\
	&\leq 2 \sup_{0\leq s\leq t}\left(e^{-\delta (t-s)}|d(s)|\right), 
\end{align}
here we have used the fact $\left|\left(\frac{|\lambda_n|}{\lambda_n}(e^{\lambda_nt}-1)+\frac{|\mu_n|}{\mu_n}(e^{\mu_nt}-1)
	\right)\right|<2$ as $\lambda_n, \mu_n <0$.

\noindent
Therefore combining \eqref{un} and \eqref{i}, we have 
\begin{equation}\label{un_final}
	|u_n(t)|^2\leq C_1e^{-2\delta t}[|u_{0,n}|^2+|z_{0,n}|^2]+  M_1 \left({\sup_{0\leq s\leq t}\left(e^{-\delta (t-s)}|d(s)|\right)}\right)^2.
\end{equation}
Here, we have used that $\frac{(|\lambda_n|^2+|\mu_n|^2)}{|\mu_n-\lambda_n|^2}, \frac{2}{|\mu_n-\lambda_n|^2} \text{and }\frac{ n^2\pi^2}{|\mu_n-\lambda_n|^2} $ are bounded from the expression of $\mu_n, \lambda_n$ and $M_1=\frac{2\pi^2}{{(\lambda+\pi^2-\delta)^2}-4\rho \gamma}.$
\paragraph{\underline{\it{Estimate of $z_n(t)$}}}
\begin{align}\label{zn}
	\nonumber&	|{z_n(t)}|^2\leq \frac{e^{-2\delta t}}{|\mu_n-\lambda_n|^2}\bigg[|u_{0,n}|^2\gamma^2(e^{\lambda_n t}- e^{\mu_n t})^2+|z_{0,n}|^2(\lambda_ne^{\mu_n t}- \mu_n e^{\lambda_n t})^2\bigg]\\
	\nonumber	& \hspace{5cm}+\frac{ n^2\pi^2}{|\mu_n-\lambda_n|^2}{\left(\int_{0}^{t}e^{-\delta (t-s)}\left(  e^{\lambda_n (t-s)}-e^{\mu_n (t-s)}\right)d(s) ds\right)^2}\\
	&\leq C {e^{-2\delta t}}\bigg[|z_{0,n}|^2 \frac{(|\lambda_n|^2+|\mu_n|^2)}{|\mu_n-\lambda_n|^2}+|u_{0,n}|^2\frac{2}{|\mu_n-\lambda_n|^2}\bigg]+\frac{ n^2\pi^2}{|\mu_n-\lambda_n|^2} \mc J^2.
\end{align}
Let $\epsilon \in (0,\delta)$ be given. We estimate the term $\mc J=\int_{0}^{t}e^{-\delta (t-s)}\left(  e^{\lambda_n (t-s)}-e^{\mu_n (t-s)}\right)d(s) ds$ as follows
\begin{align}\label{j}
	\nonumber	|\mc J|&\leq \sup_{0\leq s\leq t}\left(e^{-(\delta-\epsilon) (t-s)} |d(s)|\right)\int_{0}^{t}\left|\left(e^{(\lambda_n-\epsilon) (t-s)}- e^{(\mu_n-\epsilon) (t-s)}\right)\right|ds\\
	\nonumber	&\leq \sup_{0\leq s\leq t}\left(e^{-(\delta-\epsilon) (t-s)}|d(s)|\right)\left(e^{(\lambda_n-\epsilon)t}\int_{0}^{t}e^{-(\lambda_n-\epsilon)s}+e^{(\mu_n-\epsilon)t}\int_{0}^{t}e^{-(\mu_n-\epsilon)s}\right)\\
	\nonumber	&\leq \sup_{0\leq s\leq t}\left(e^{-(\delta-\epsilon) (t-s)}|d(s)|\right)\left(\frac{e^{(\lambda_n-\epsilon)t}-1}{(\lambda_n-\epsilon)}+\frac{e^{(\mu_n-\epsilon)t}-1}{(\mu_n-\epsilon)}
	\right)\\
	&\leq M_2 \sup_{0\leq s\leq t}\left(e^{-(\delta-\epsilon) (t-s)}|d(s)|\right),
\end{align}
where $M_2=\sup\limits_{t>0}\left(\frac{e^{(\lambda_n-\epsilon)t}-1}{(\lambda_n-\epsilon)}+\frac{e^{(\mu_n-\epsilon)t}-1}{(\mu_n-\epsilon)}
	\right).$
Therefore combining \eqref{zn} and \eqref{j}, we have 
\begin{equation}\label{zn_final}
	|z_n(t)|^2\leq C_1e^{-2\delta t}[|u_{0,n}|^2+|z_{0,n}|^2]+  M_3 \left({\sup_{0\leq s\leq t}\left(e^{-(\delta-\epsilon) (t-s)}|d(s)|\right)}\right)^2,
\end{equation}
where $M_3=M_1 M_2.$ Using the expressions of the eigenvalues and noting the fact that $\lambda_n, \mu_n<0,$ we have $M_2\leq \bigg[\frac{1}{\epsilon}+\frac{1}{\pi^2+\lambda-\delta+\epsilon} \bigg].$
Adding \eqref{un_final} and \eqref{zn_final}, the required result follows. 
\end{proof}
\begin{rem}\label{optimal bd}
    The constant $\vartheta$ in the input-to-state stability estimate \Cref{lmm iss} is not sharp. The reason for getting such a bound is that we need to use an upper bound of the term $M_2$ in the estimate of $z_n,$ see \eqref{zn}. In other cases, such as for scalar parabolic equations, an explicit bound for $\vartheta$ can be derived, depending on the system parameters and the decay rate. For further details, see \cite[Lemma 3]{espitia2021event}.
\end{rem}
Now we are in a position to prove our main exponential stabilization result.
\begin{theorem}\label{main_theorem}
 Let $\epsilon \in (0,\delta)$ be given, $\vartheta$ be as in \Cref{lmm iss} and $\beta$ be chosen in such a way that\begin{equation}\label{beta}
     \Phi_e:= 2 \beta  \vartheta \Vert k(1,\cdot) \Vert \Vert \Pi^{-1} \Vert<1.
 \end{equation} Then, the  closed-loop system  \eqref{FHNlin ET}  with event-triggered boundary control    \eqref{tr rule}-\eqref{controlfunction} has a unique solution and is globally exponentially stable, i.e. there exists $M>0$ depending on $\epsilon$ and the system parameters such that  for every $(v_0,w_0) \in L^2(0,1)\times L^2(0,1)$ the unique solution $(v,w) \in  \mc {C}^{0}([0,\infty); L^{2}(0,1)\times L^2(0,1))$  of \eqref{FHNlin ET}
satisfies
	\begin{equation}\label{ex stb}
		\Vert v(t) \Vert+\Vert w(t) \Vert \leq M e^{-(\delta-\epsilon) t} \left(\Vert v(0) \Vert+\Vert w(0) \Vert\right), \quad \text{for all} \quad t\geq 0.
	\end{equation}
\end{theorem}
\begin{proof}
For given $\epsilon\in(0,\delta)$, we set $\vartheta$ as in \Cref{lmm iss} and, to abridge the notation, we denote $\nu=\delta-\epsilon$. It follows from \eqref{triggering}  that the following inequality holds for all $t \in [t_{j},t_{j+1})$:
	\begin{equation}\label{est2}
		\vert d(t) \vert \leq \beta \Vert k(1,\cdot) \Vert  V(t_j)  + \beta \Vert k(1,\cdot) \Vert  V(t), 
	\end{equation}
for every $t\geq 0,$ we define $f(t)=\max\{t_j: j\geq 0, t\geq t_j\}.$ Then \eqref{est2} implies the following 
\begin{equation}\label{est3}
	\vert d(t) \vert \leq \beta \Vert k(1,\cdot) \Vert  V(f(t))  + \beta \Vert k(1,\cdot) \Vert  V(t) .
\end{equation} 
	Therefore, from \eqref{est3}, we deduce the following inequality for all $t\geq 0$:
	
	\begin{equation}\label{est4}
		\left( \vert d(t)\vert  e^{\nu t}  \right) \leq 2 \beta \Vert k(1,\cdot) \Vert \sup_{0 \leq s \leq t} \left(  V(s)  e^{\nu s} \right) .
	\end{equation}
Let us choose any $\hat t\geq 0$. Then for all $t\in [0,\hat{t}]$ we have from \eqref{est4}
\begin{equation}\label{est5}
 \left( \vert d(t)\vert  e^{\nu t}  \right) \leq 2 \beta \Vert k(1,\cdot) \Vert \sup_{0 \leq s \leq \hat t} \left(  V(s)  e^{\nu s} \right). 
\end{equation}
As the above estimate holds for all $t\in [0,\hat t]$, we have
\begin{equation}
\sup_{0 \leq s \leq \hat t}	 \left( \vert d(s)\vert  e^{\nu s}  \right) \leq 2 \beta \Vert k(1,\cdot) \Vert \sup_{0 \leq s \leq \hat t} \left(  V(s)  e^{\nu s} \right). 
\end{equation}
Since $\hat t$ is arbitrary, one can write the following
\begin{equation}\label{est6}
	\sup_{0 \leq s \leq t}	 \left( \vert d(s)\vert  e^{\nu s}  \right) \leq 2 \beta \Vert k(1,\cdot) \Vert \sup_{0 \leq s \leq  t} \left(  V(s)  e^{\nu s} \right). 
\end{equation}
	On the other hand, by \Cref{lmm iss}, we obtain
	\begin{equation}\label{ISS_target}
	\left(	\Vert u(t) \Vert+\Vert z(t) \Vert\right) e^{\nu t} \leq C\left(	\Vert u(0) \Vert+\Vert z(0) \Vert\right) + \vartheta \sup_{0\leq s \leq t}\left(\vert d(s) \vert e^{\nu s}\right).
	\end{equation}
	Again let us pick $\hat t\geq 0$. Then for all $t\in [0, \hat t],$ we have \begin{equation}\label{est7}
	\left(	\Vert u(t) \Vert+\Vert z(t) \Vert\right) e^{\nu t} \leq C\left(	\Vert u(0) \Vert+\Vert z(0) \Vert\right) + \vartheta \sup_{0\leq s \leq \hat t}\left(\vert d(s) \vert e^{\nu s}\right).
\end{equation}
As the above estimate holds for all $t\in [0,\hat t]$, we have
\begin{equation}\label{est8}
\sup_{0\leq s \leq \hat t}	\left(	\Vert u(s) \Vert+\Vert z(s) \Vert\right) e^{\nu s} \leq C\left(	\Vert u(0) \Vert+\Vert z(0) \Vert\right) + \vartheta \sup_{0\leq s \leq \hat t}\left(\vert d(s) \vert e^{\nu s}\right).
\end{equation}
Since $\hat t$ is arbitrary we can write the following
		\begin{equation}\label{est9}
		\sup_{0\leq s \leq t}\left( 	\left(	\Vert u(s) \Vert+\Vert z(s) \Vert\right)  e^{\nu s} \right) \leq C \left(	\Vert u(0) \Vert+\Vert z(0) \Vert\right) + \vartheta \sup_{0\leq s \leq t}\left(\vert d(s) \vert e^{\nu s}\right).
	\end{equation}
	Hence, combining \eqref{est6} with  \eqref{est9}, we obtain 
	\begin{equation*}
		\sup_{0\leq s \leq t}\left( 	\left(	\Vert u(s) \Vert+\Vert z(s) \Vert\right)  e^{\nu s} \right) \leq C \left(	\Vert u(0) \Vert+\Vert z(0) \Vert\right) + 2 \beta  \vartheta \Vert k(1,\cdot) \Vert \sup_{0 \leq s \leq t} \left(  V(s)  e^{\nu s} \right) 
	\end{equation*}
	and using the fact $  V(s) \leq {\Pi}^{-1} \left(\Vert u(t) \Vert+\Vert z(t) \Vert\right) $, we deduce
	\begin{equation*}\label{estimateofsup-d3}
		\sup_{0\leq s \leq t}\left( \left(\Vert u(s) \Vert+\Vert z(s) \Vert  \right)e^{\nu s} \right) \leq C \left(	\Vert u(0) \Vert+\Vert z(0) \Vert\right) + \Phi_e \sup_{0 \leq s \leq t} \left( \left(\Vert u(s) \Vert+\Vert z(s) \Vert\right) e^{\nu s} \right) 
	\end{equation*}
	where
		$\Phi_e:= 2 \beta  \vartheta \Vert k(1,\cdot) \Vert \Vert \Pi^{-1} \Vert.$
 Thereby, thanks to \eqref{beta} and the invertibility of the backstepping transformation, the solution  to the closed-loop system \eqref{FHNlin ET}  with event-triggered control  \eqref{triggering}-\eqref{controlfunction} satisfies
	\begin{equation*}
		\sup_{0 \leq s \leq t} \left( \left(\Vert v(s) \Vert+\Vert w(s) \Vert\right) e^{\nu s} \right)  \leq C \norm{\Pi}\norm{\Pi^{-1}}(1-\Phi_e)^{-1}  \left(\Vert v(0) \Vert+\Vert w(0) \Vert\right).
	\end{equation*}
	which leads to the following:
	\begin{equation*}
		 \left(\Vert v(t) \Vert+\Vert w(t) \Vert\right)    \leq  M e^{-\nu t}\left(\Vert v(0) \Vert+\Vert w(0) \Vert\right), \, \forall t\geq 0,
	\end{equation*}
	with $ M:= C \norm{\Pi}\norm{\Pi^{-1}}(1-\Phi_e)^{-1}$. The proof is finished.
	\end{proof}

\section{Numerical experiments}\label{sec:5}

In this section, we present some numerical illustrations and comments about the practical implementation of the backstepping technique and the event-triggering control. In what follows we shall use the notation $\inter{a,b}=[a,b]\cap \mathbb N$ for any real numbers $a<b$. 

\subsection{Some considerations on the system parameters}

In \Cref{setting of the problem}, we have mentioned that the linearized FitzHugh-Nagumo system \eqref{FHNlin} is exponentially stable with a decay rate $\omega=\min\{a, \delta\}$ without the action of any control. Thus the important case to study the backstepping-based stabilization problem is when $a<\delta.$ As, in this case, the free system (\eqref{FHNlin} with $q=0$) is stable with decay $e^{-a t},$ and the action of the event-triggering control law makes the decay rate up to $e^{-\omega t}$, $\omega<\delta$. 

To see the effect, in a numerical simulation, of the action of the event-triggering control law, that actually helps to stabilize an unstable system to a stable one, we choose to change the sign of the parameters $a$ in the system \eqref{FHNlin} and make the stable FHN to an unstable coupled parabolic-ODE system.
To this end, let us consider the model given by
\begin{align}\label{eq:heat_memory_simple}
	\begin{cases}
		\dis{\partial_t v} = {\pa_{x}^2 v} - a v - \rho w, & (t,x)\in (0,\infty)\times (0,1), \\
		\dis{\partial_t w} = \gamma v - \delta w, & (t,x)\in (0,\infty)\times (0,1),
		\\
		v(t,0) = 0, \quad v(t,1)=q(t) &t \in (0,\infty),
		\\
		v(0,x) = v_0(x), \quad w(0,x)=w_0(x) & x\in(0,1),
	\end{cases}
\end{align}  
where $\rho, \gamma>0$ and $\delta>0>a$. For some fixed values of $\rho, \gamma, \delta$, we choose $a <0$ in such a way that it satisfies the following
\begin{equation}\label{positve}
{(a+\pi^2-\delta)^2-4 \rho \gamma}>0 \text{ along with }
 \sqrt{{(a+\pi^2-\delta)^2}-4\rho \gamma} > (a+\pi^2+\delta).
\end{equation}
As $a$ is negative, we can always choose such $a$ with a large enough absolute value. Furthermore, the spectrum of the associated operator for the system \eqref{eq:heat_memory_simple} can be expressed as
\begin{align*}
&\lambda_n=\frac{1}{2}\bigg[-\left(a+n^2\pi^2+\delta\right) + \sqrt{\left(a+n^2\pi^2-\delta\right)^2 - 4\rho \gamma} \bigg],\\ 
&\mu_n=\frac{1}{2}\bigg[-\left(a+n^2\pi^2+\delta\right) - \sqrt{\left(a+n^2\pi^2-\delta\right)^2 - 4\rho \gamma} \bigg].
\end{align*}
These expressions suggest that it is reasonable to find the conditions for $a$ for which the first eigenvalue for $\lambda_n$ is positive. Thanks to the assumption \eqref{positve}, we have $\lambda_1>0$, which essentially ensures the instability of the system \eqref{eq:heat_memory_simple}.

It is worth mentioning that, here, we will not change the sign of the ODE component in the ODE of the system \eqref{FHNlin}. Otherwise, it is impossible to make the system exponentially stable, as indeed, the continuous backstepping-based feedback law also could not change the decay rate of the ODE, see \cite{chowdhury2024local}.

\subsection{Discretization of the coupled model and numerical implementation of the backstepping control}

For the numerical tests, system \eqref{eq:heat_memory_simple} is discretized in time by using a standard implicit Euler scheme and discretized in space by a usual finite-difference scheme. More precisely, let $N,M\in\mathbb N^*$, we set $\delta t=T/M$ and $h=1/(N+1)$ and consider the following uniform discretization for the space and time variables
\begin{align*}
	&0=x_0<x_1<\ldots< x_{N}<x_{N+1}=1, 
	\\
	&0=t_0<t_1<\ldots< T_{M}=T,
\end{align*}
where $x_i=i h$, $i\in\inter{0,N+1}$, and $t_n=n\delta t$, $n\in\inter{0,M}$. The numerical approximation of a function $f=f(x,t)$ at a grid point $(t_n,x_i)$ will be denoted as $f_i^n:=f(t_n,x_i)$ and, for fixed $n$, we write $f^n=\left(\begin{array}{ccc}f_1^n,  \dots, f_N^n\end{array}\right)^\top$ the evaluation at the interior points.

Defining $Z:=(v,w)^\top$ and using the above notation, the fully-discrete version of \eqref{eq:heat_memory_simple} takes the form
\begin{equation}\label{num_scheme}
\begin{cases}
\displaystyle \frac{Z^{n+1}- Z^{n}}{\delta t}+\mathcal A_h Z^{n+1}=\mathcal B_h q^{n+1}, \quad n\in\inter{0,M-1}\\
Z^0=Z_h^0,
\end{cases}
\end{equation}
where $(Z^n)_{n\in\inter{0,M}}\subset \mathbb R^{2}\otimes \mathbb R^{N}$ and $(q^n)_{n\in\inter{1,M}}\subset \mathbb R$ are the (discrete) state and control variables, respectively, $Z_h^0\in \mathbb R^{2}\otimes \mathbb R^{N}$ is the approximation of the initial data $(v_0,w_0)$, the (boundary) control operator $\mathcal B_h\in{\mathbb R^2\otimes R^N}$ is given by
\begin{equation*}
\mathcal B_h=\frac{1}{h^2}\left(\begin{array}{c}1 \\ 0\end{array}\right)\otimes \left(\begin{array}{c}0 \\0 \\\vdots \\1\end{array}\right) _{N},
\end{equation*}
and $\mathcal A_h\in \mathbb R^{2\times 2}\otimes\mathbb R^{N\times N}$ is
\begin{equation*}
\mathcal A_h= \left(\begin{array}{cc} 1 & 0 \\ 0& 0 \end{array}\right)\otimes \mathcal A_{h,D} + \mathsf C \otimes I_{N\times N},
\end{equation*}
where $\mathcal A_{h,D}\in \mathbb R^{N\times N}$ is the usual tridiagonal matrix coming from the discretization of the Laplacian operator $-\partial_{xx}$ with homogeneous Dirichlet boundary conditions, that is, $(\mathcal A_h y)_{i}=-\frac{1}{h^2}(y_{i+1}-2y_{i} +y_{i-1})$, $i=\inter{1,N}$ and $\mathsf C=\left(\begin{array}{cc} a & \rho \\ -\gamma & \delta \end{array}\right)$ are the coupling coefficients.

Now, let us turn our attention to the control part. Recall that the backstepping control is given by the explicit feedback control law
\begin{equation}\label{control_num}
    q(t) = \int_0^1 k(1,y)v(t,y) \, dy
\end{equation}
where \( k = k(x,y) \) is the solution to \eqref{kernel}. This particular structure simplifies the implementation of the control. In fact, we can approximate \eqref{control_num} directly at a time-grid point \( t_n \) (i.e., \( q(t_n) = q^n \)) with the formula 
\begin{equation}\label{eq:kernel_discrete}
    q^n = h \sum_{i=1}^{N} k(1,x_i) v_{i}^{n}.
\end{equation}

\begin{rem}
There are, of course, different ways to discretize the integral in \eqref{control_num}. For instance the trapezoidal rule would slightly modify \eqref{eq:kernel_discrete}, specifically,  
\begin{equation*}
    q^n = \frac{h}{1 - \frac{h\varpi}{2}} \sum_{i=1}^{N} k(1,x_i)v_i^n, \quad n \in \inter{1,M}.
\end{equation*}
where \( \varpi = a - {\lambda} < 0 \). We opted for \eqref{eq:kernel_discrete} due to its simplicity in numerical implementation. 
\end{rem}

Note that \eqref{eq:kernel_discrete} can be written in compact form and using the complete state variable $Z$ by introducing the vector $K\in\mathbb R^{2}\otimes \mathbb R^{N}$ defined by \begin{equation}\label{KK}
K^\top:=\left(\begin{array}{cccccccc} k_{1,1} & k_{1,2} &\dots & k_{1,N} & 0 & 0 & \cdots & 0 \end{array}\right)
\end{equation} where $k_{1,i}:=k(1,x_i)$. More precisely, 
\begin{equation}\label{eq:feedback_control}
q^n= {h} K^{\top}Z^n, \quad n\in\inter{1,M}.
\end{equation}
Thus, combining \eqref{num_scheme} and \eqref{eq:feedback_control}, the feedback control system simplifies to

\begin{equation}\label{num_scheme_feed}
\begin{cases}
    \displaystyle \frac{Z^{n+1} - Z^n}{\delta t} + \left( \mathcal{A}_h - {h} \mathcal{B}_h K^\top \right) Z^{n+1} = 0, \quad n \in \inter{0,M-1}, \\
    Z^0 = Z_h^0.
\end{cases}
\end{equation}

\subsubsection{Numerical illustration of the backstepping control}

To illustrate the behavior of the controlled system \eqref{eq:heat_memory_simple}, let us choose the following parameters: we set 
\begin{equation}\label{param}
T=6, \quad  v_0(x)=\sin(\pi x), \quad w_0(x)=\sin(2\pi x),
\end{equation} 
and the coupling matrix 
\begin{equation}\label{coupling}
\mathsf C=\left(\begin{array}{cc}-11 & 1 \\ -1 & 1 \end{array}\right),
\end{equation}
where we choose the entries in the matrix $\mathsf C$ according to the constraints \eqref{positve}.

With these parameters, it can be verified that the uncontrolled system (setting $q\equiv 0$) is unstable, as shown in \Cref{fig:uncontr_fig}. Such figures have been obtained with the help of the numerical scheme \eqref{num_scheme_feed} with parameters $N=40$, $M=2000$\footnote{These simulation parameters will be used in the remainder of the simulations shown in this section. } and setting $\mathcal B_h\equiv 0$ for observing the uncontrolled dynamics. 
\begin{figure}[htbp!]
    \centering
    \begin{subfigure}{0.35\textwidth}
        \includegraphics[width=\textwidth]{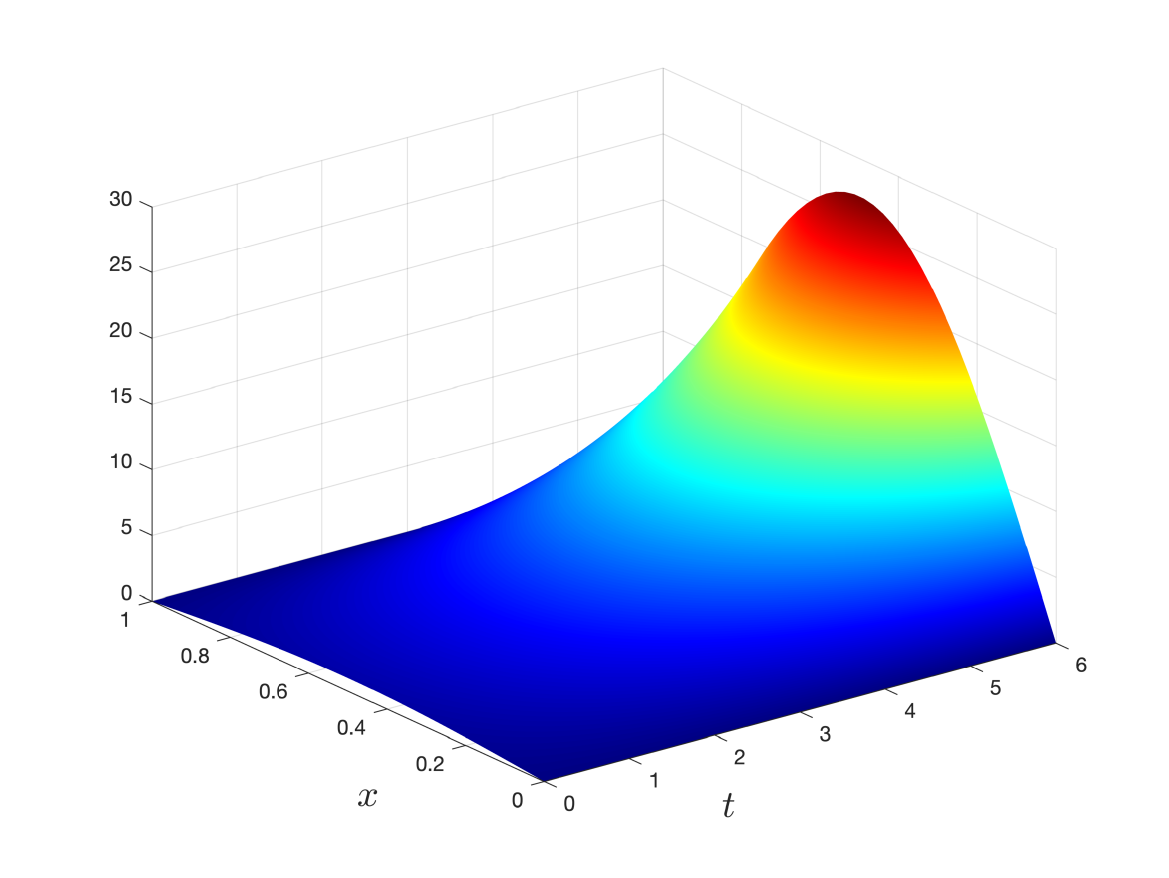} 
        \caption{The PDE component $(t,x)\mapsto v(t,x)$.}
        \label{fig:sub1}
    \end{subfigure}
    \qquad 
    \begin{subfigure}{0.35\textwidth}
        \includegraphics[width=\textwidth]{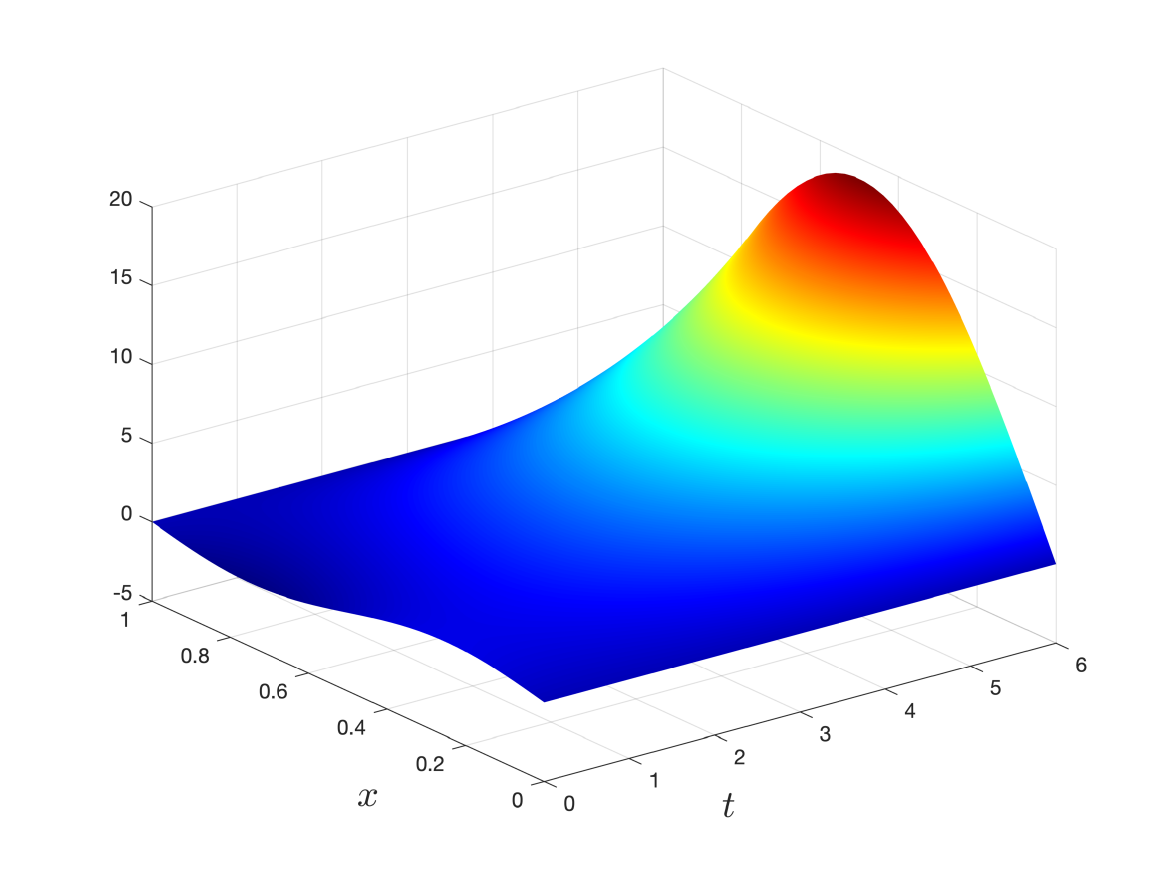} 
        \caption{The ODE component $(t,x)\mapsto w(t,x)$.}
        \label{fig:sub2}
    \end{subfigure}
    \caption{Evolution in time of the uncontrolled system \eqref{eq:heat_memory_simple}.}
    \label{fig:uncontr_fig}
\end{figure}

Now, we choose ${\lambda}=1$, compute $K$ as in \eqref{KK} with the help of formula \eqref{kernel bes} and use the numerical scheme \eqref{num_scheme_feed} to illustrate the behavior of the closed-loop system, see \Cref{fig:contr_system}. We see that by applying the boundary control shown in \Cref{subcontrol} the system becomes stable.

\begin{figure}[htbp!]
    \centering
    \begin{subfigure}{0.35\textwidth}
        \includegraphics[width=\textwidth]{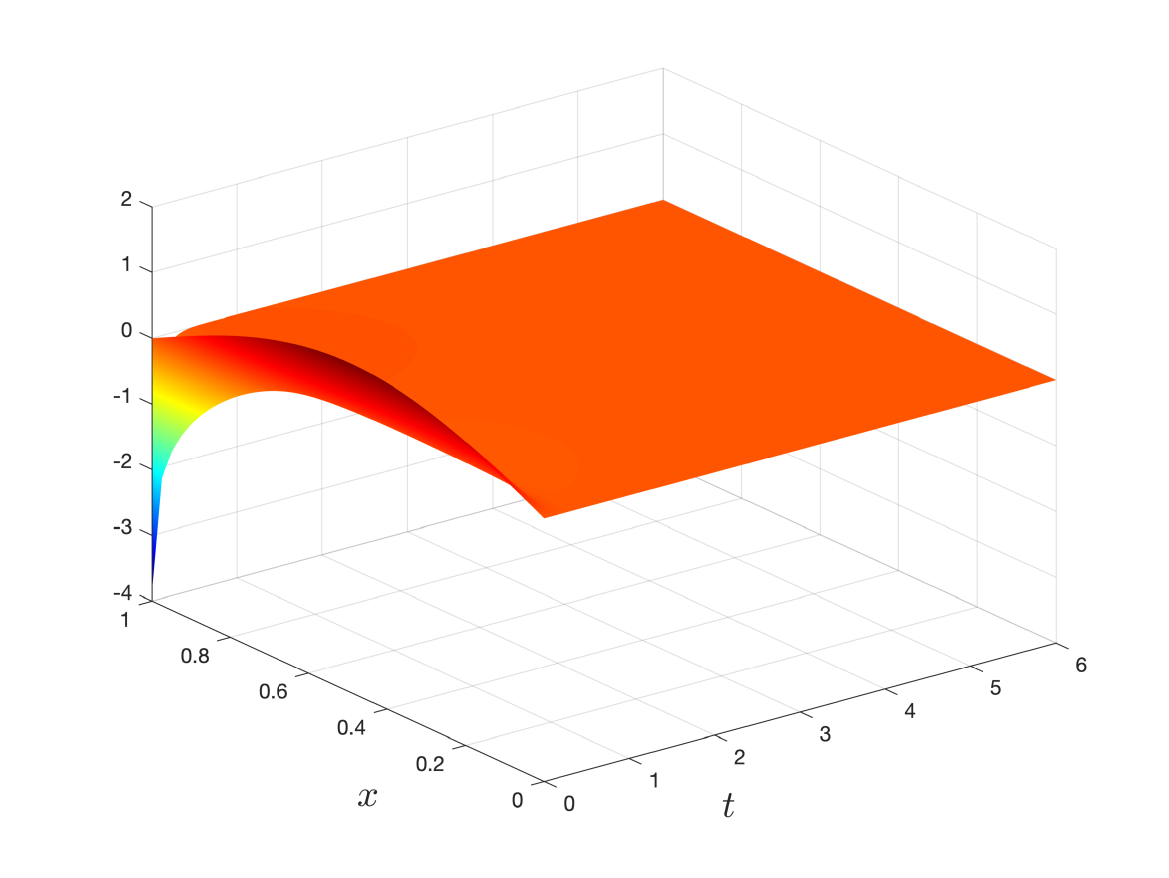} 
        \caption{The PDE component $(t,x)\mapsto v(t,x)$.}
        \label{fig:sub1_ctr}
    \end{subfigure}
    \qquad 
    \begin{subfigure}{0.35\textwidth}
        \includegraphics[width=\textwidth]{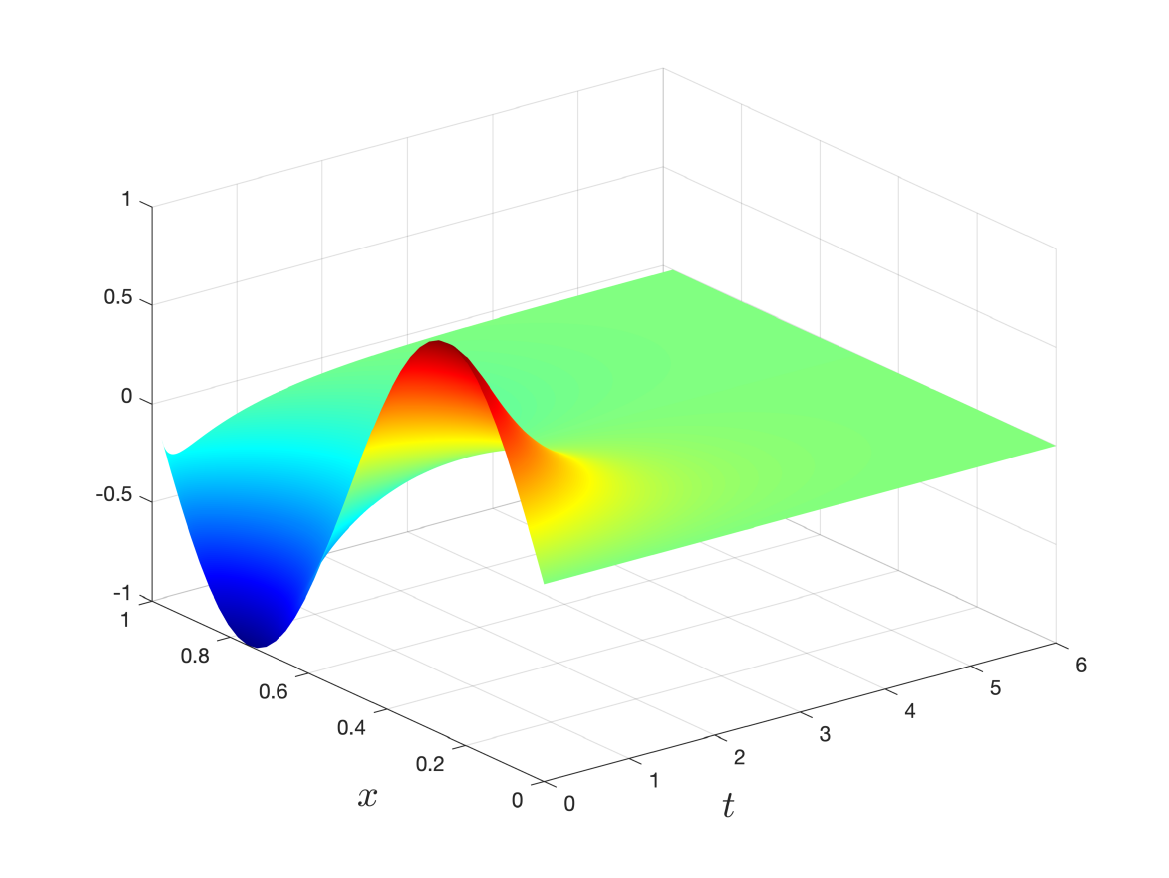} 
        \caption{The ODE component $(t,x)\mapsto w(t,x)$.}
        \label{fig:sub2_ctr}
    \end{subfigure}   \\
    \vspace{0.8cm} 
    \begin{subfigure}{0.5\textwidth}\centering
      	\includegraphics{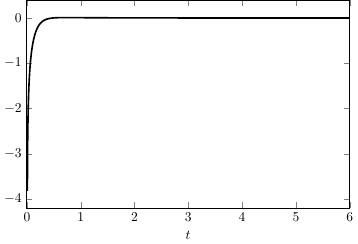} %
\caption{The control $t\mapsto q(t)$.}\label{subcontrol}
    \end{subfigure} 
    \caption{Evolution in time of the controlled system \eqref{eq:heat_memory_simple} with backstepping control \eqref{control_num}.}
    \label{fig:contr_system}
\end{figure}

Unlike other similar backstepping control problems, as noted in \cite{chowdhury2024local}, the best achievable exponential rate for our closed-loop system with backstepping control is determined by the decay of the ODE component in \eqref{eq:heat_memory_simple}. In other words, we cannot attain an arbitrary exponential decay rate regardless of the control design. To illustrate this, we have computed feedback controls with different design parameters ${\lambda}$, and compared the dynamics, as shown in \Cref{fig:comparison_lambda}. Even though the dynamics are distinct at the very beginning of the time interval, after some point, all dynamics converge and become almost identical. The figure effectively illustrates this convergence behavior.

\begin{figure}[htbp]
	\centering
	\includegraphics{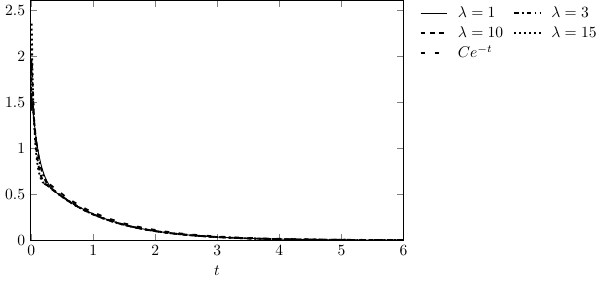} %
\caption{Norm of the closed-loop dynamics $t\mapsto \|v(t)\|_{L^2}+\|w(t)\|_{L^2}$ of \eqref{eq:heat_memory_simple} with different design parameters $ \lambda$. For comparison, we have added the exponential function $Ce^{-t}$ for some $C>0$, which is the best theoretical decay rate for system \eqref{eq:heat_memory_simple} with coupling parameters \eqref{coupling}. }
\label{fig:comparison_lambda}
\end{figure}

\subsection{Numerical illustration of the event-triggering control}

With the discrete system given in \eqref{num_scheme} and the backstepping control law defined by \eqref{eq:feedback_control}, we can easily adapt our computation tool to address the event-triggering problem. For clarity and future reference, we present a brief pseudocode in \Cref{alg:event_triggered_backstepping} that outlines the essential steps. 

Using such algorithm, along with the system parameters specified in \eqref{param} and \eqref{coupling}, we aim to stabilize the free dynamics via event-triggered control. In this context, let $\epsilon=0.05$, then the parameter $\beta$ in \Cref{main_theorem} must be carefully selected. 

A quick computation shows that $\vartheta=30.0206$ according to \eqref{vartheta}, $\|\Pi^{-1}\| = 2.2302$ and $\|k\| = 6.5968$ if $\lambda=1$, so choosing $\beta = 0.001$ ensures \eqref{beta}. In \Cref{fig:contr_system_events}, we illustrate the time evolution of the system under event-triggered control. Notably, the small magnitude of $\beta$ leads to frequent triggering, resulting in a control strategy that closely resembles continuous control (compare \Cref{details} and \Cref{subcontrol}). However, upon closer inspection, we can observe distinct triggers, with the control remaining constant over certain time intervals. Also, the similitude in the control translates into a nearly indistinguishable behavior in the system's states.

\begin{figure}

\begin{subfigure}{0.35\textwidth}
        \includegraphics[width=\textwidth]{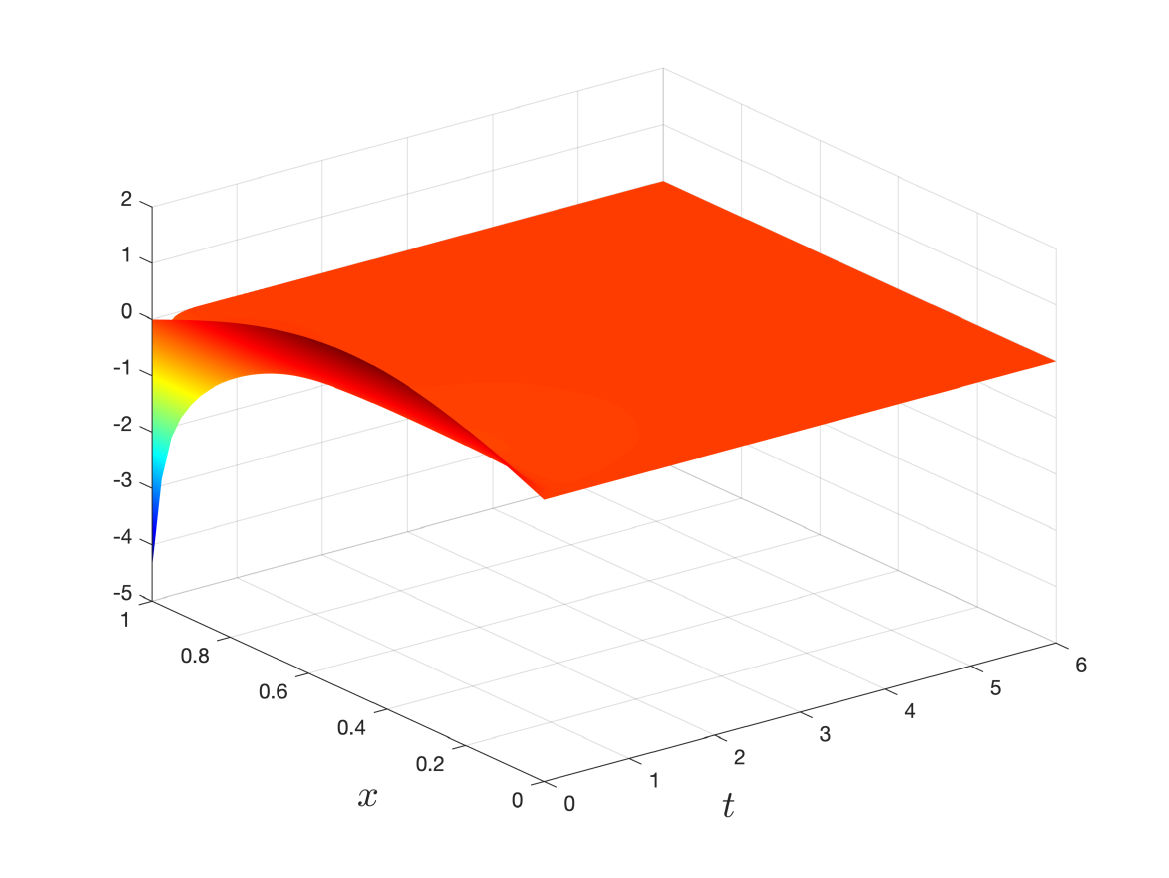} 
        \caption{The PDE component $(t,x)\mapsto v(t,x)$.}
        \label{fig:sub1_event}
    \end{subfigure}
    \qquad 
    \begin{subfigure}{0.35\textwidth}
        \includegraphics[width=\textwidth]{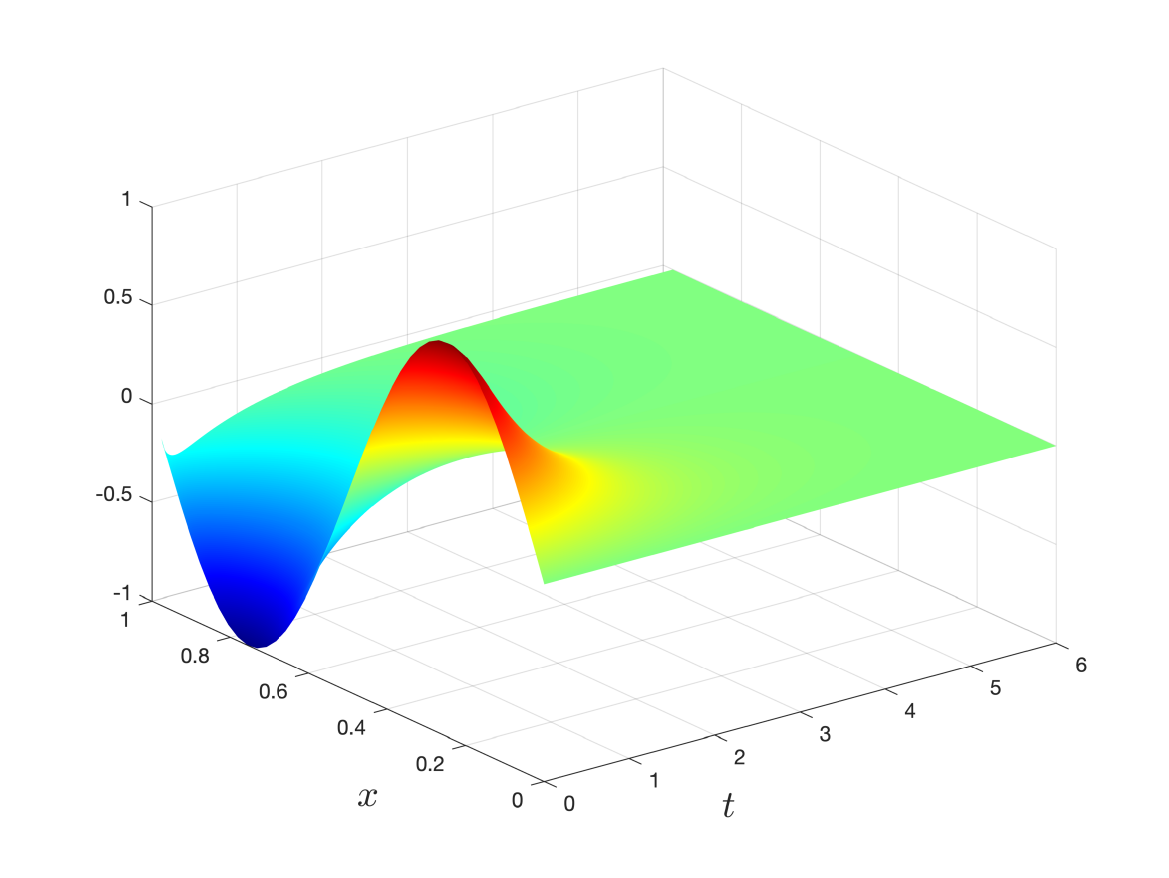} 
        \caption{The ODE component $(t,x)\mapsto w(t,x)$.}
        \label{fig:sub2_event}
    \end{subfigure}   \\
    \vspace{0.8cm} 
    \begin{subfigure}{0.55\textwidth}
    \includegraphics{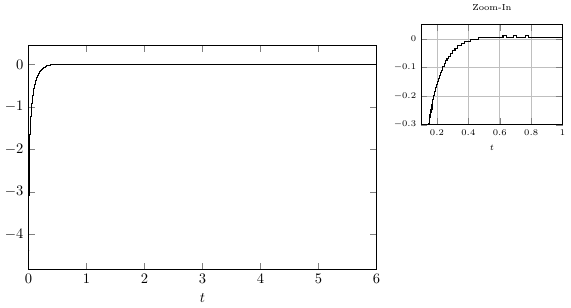} %
\caption{Event-triggered control and zoom-in.}\label{details}
\end{subfigure}
\caption{Evolution in time of the controlled system \eqref{eq:heat_memory_simple} with event-triggered control with design parameter $\beta=0.001$.}
    \label{fig:contr_system_events}
\end{figure}

We conclude this section by noting that the parameter $\beta$ can actually be overestimated. From \Cref{optimal bd}, note that we are not providing an optimal expression for the parameter $\vartheta$, but only an upper bound. In \Cref{fig:control_trigger_less}, we present another event-triggered control, this time with $\beta = 0.05$, which does not satisfy the conditions of \Cref{main_theorem}. It can be observed that this control triggers far fewer times, especially at the beginning of the time interval, yet still achieves stabilization of the entire system. The similarity of the controlled states to those in previous experiments remains very close, and to avoid redundancy, we omit additional figures.

%
%

\begin{figure}[htbp!]
	\centering
	    \includegraphics{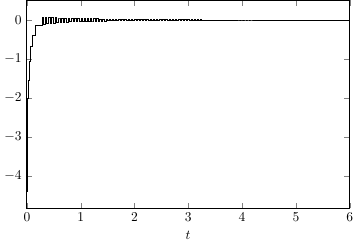} %
	\caption{Event-triggered control with design parameter $\beta=0.05$.}
\label{fig:control_trigger_less}
\end{figure}

\begin{algorithm}
\caption{Event-triggered control with backstepping}\label{alg:event_triggered_backstepping}
\begin{algorithmic}[1]
\State \textbf{Input:} Initial (discrete) states $(v_0,w_0)$; discrete parameters $h$, $\delta t$ (resp. $N$, $M$); matrices $\mathsf C$, $\mathcal A_h$, $\mathcal B_h$, $K$; time $T>0$, parameter $\beta$
\State \textbf{Initialize:} 
\State \hskip1.5em Set $v^0 \gets u_0$, $w^0 \gets w_0$ and arrange in vector form  $Z^0 \gets \left(\begin{array}{c} v^0 \\ w^0\end{array}\right)$
\State \hskip1.5em Set $z(t_0) \gets Z^0$
\State \hskip1.5em Compute norm: $\|K\|\gets \sqrt{h K^\top K}$
\For{$i = 0$ to $M-1$}
    \State Update state with \eqref{num_scheme} and control $q^{i+1}=h K^Tz(t_i)$: $$Z^{i+1} \gets (I+\delta t\mathcal A_h)^{-1} \left(Z^{i} + \delta t \mathcal B_h z(t_i)  \right)$$
    \State Set $z(t) \gets Z^{i+1}$
    
    \State Compute difference: $d(t) \gets h K^T (z(t_i) - z(t))$
    
    \State Split the vectors in components: $\left(\begin{array}{c} v(t) \\ w(t)\end{array}\right) \gets z(t)$, $\left(\begin{array}{c} v(t_i) \\ w(t_i)\end{array}\right) \gets z(t_i)$
    
    \State Compute norms: $\|v(t)\|\gets \sqrt{h v(t)^\top v(t)}$ (similarly for variables $w(t)$, $v(t_i)$, $w(t_i)$)
    
    \State Compute triggering threshold:
    \[
    \text{quotient} \gets \beta \|K\| \left( \|v(t)\|+\|w(t)\|+\|v(t_i)\|+\|w(t_i)\|\right)
    \]
    
    \If{$|d| > \text{quotient}$}
        \State Update variable: $z(t_{i+1}) \gets Z^{i+1}$
    \Else 
    	\State Keep variable: $z(t_{i+1})\gets z(t_i)$
    \EndIf
\EndFor
\State \textbf{Output:} Final controlled solution $Z=(Z^n)_{n\in\inter{0,M}}$ and event-triggered control $q=(q^n)_{n\in\inter{1,M}}$.
\end{algorithmic}
\end{algorithm}

\section*{Acknowledgements}

Luz de Teresa is grateful to Lucie Baudouin for introducing her to the event-triggering strategies and their applications in the control of PDEs.

\bibliographystyle{alpha}
\bibliography{mybibfile}
\end{document}